\documentclass[12pt,leqno]{amsart}  
\usepackage{amsmath,amstext,amsthm,amssymb,amsxtra}
\usepackage[top=1.5in, bottom=1.5in, left=1.25in, right=1.25in]	{geometry}
\usepackage{txfonts} 
\usepackage[T1]{fontenc}
\usepackage{lmodern}

 \usepackage{euler}   
\usepackage{pdfsync}

\usepackage{float} 

\usepackage{tikz}

\allowdisplaybreaks

\usepackage{amssymb,amsmath,amsthm,mathrsfs,enumerate,multicol,esint}

\usepackage{mathtools}
\mathtoolsset{showonlyrefs,showmanualtags}

\usepackage{hyperref} 
\hypersetup{
    colorlinks=true,       
    linkcolor=blue,          
    citecolor=magenta,        
    filecolor=magenta,      
    urlcolor=cyan           
}

\usepackage[msc-links]{amsrefs}  

\theoremstyle{plain} 
\newtheorem{lemma}[equation]{Lemma} 
\newtheorem{proposition}[equation]{Proposition} 
\newtheorem{theorem}[equation]{Theorem} 
\newtheorem{corollary}[equation]{Corollary} 

\newtheorem{example}[equation]{Example}

\newtheorem*{TheoremA}{Theorem A}
\newtheorem*{TheoremB}{Theorem B}

\theoremstyle{definition}

\theoremstyle{remark}
\newtheorem{remark}[equation]{Remark}

\newcommand{\ContainC}[1]{\Bigl( #1 \Bigr)}

\newcommand{\RH}{RH} 

\newcommand{\f}{\frac}

\def\barint{\kern4pt
\raise3.4pt\hbox{\vrule height.8pt width5pt}%
\kern-9pt 
\int}

\numberwithin{equation}{section}

\newcommand{\unit}{1\!\!1}

\newcommand{\norm}[1]{\ensuremath{\left\|#1\right\|}}
\newcommand{\abs}[1]{\ensuremath{\left\vert#1\right\vert}}

\newcommand{\R}{\mathbb{R}}

\newcommand{\avg}[1]{\langle #1 \rangle}

\newcommand{\bmo}{\textnormal{BMO}}

\newcommand{\BMO}{\textnormal{BMO}}

\newcommand{\bp}[1]{\widetilde{\psi}_{#1}}

\newcommand{\rh}{\textnormal{RH}}

\newcommand{\al}{\alpha}

\newcommand{\pp}[1]{\psi_{#1}}

\title[Quantitative Estimates in the Schr\"odinger Setting] {$A_p$ weights and Quantitative Estimates in the Schr\"odinger Setting}

\author[J. Li]{Ji Li}
\address{J. Li, Department of Mathematics, Macquarie University, NSW, 2019, Australia} 
\email{ji.li@mq.edu.au}

\author[R. Rahm]{Robert Rahm}
\address{R. Rahm, Department of Mathematics, Washington University - St. Louis, St. Louis, MO 63130-4899 USA} 
\email{rahm@math.wustl.edu}

\author[B.D. Wick]{Brett D. Wick}
\address{B. D. Wick, Department of Mathematics, Washington University - St. Louis, St. Louis, MO 63130-4899 USA} 
\email{wick@math.wustl.edu}

\thanks{B.D. Wick's research is supported by National Science 
Foundation grant DMS \# 1560955.  J. Li's research supported in part by 
ARC DP 160100153 and a Macquarie University New Staff Grant.}

\begin{document}

\begin{abstract}
Suppose $L=-\Delta+V$ is a Schr\"odinger operator on $\mathbb{R}^n$ with a potential 
$V$ belonging to certain reverse H\"older class $RH_\sigma$ with $\sigma\geq n/2$. 
The aim of this paper is to study the $A_p$ weights associated to $L$, denoted 
by $A_p^L$, which is a larger class than the classical Muckenhoupt $A_p$ weights.   We first prove the quantitative $A_p^L$ bound 
for the maximal function and the maximal heat semigroup associated to $L$. Then we 
further provide the quantitative $A_{p,q}^L$ bound for the fractional integral 
operator associated to $L$. We point out that all these quantitative bounds are known 
before in terms of the classical $A_{p,q}$ constant. However, since $A_{p,q}\subset 
A_{p,q}^L$, the $A_{p,q}^L$ constants are smaller than $A_{p,q}$ constant. Hence, 
our results here provide a better quantitative constant for maximal functions 
and fractional integral operators associated to $L$. Next, we prove 
two--weight inequalities for the fractional integral operator; these 
have been unknown up to this point. 
Finally we 
also have a study on the ``exp--log'' link between $A_p^L$ and $BMO_L$ (the BMO space 
associated with $L$), and show that for $w\in A_p^L$, $\log w$ is in  $BMO_L$, and that
the reverse is not true in general.


\end{abstract}

	\maketitle  
	


\section{Introduction and Statement of Main Results} \label{s:1}
\setcounter{equation}{0}


The theory of Muckenhoupt $A_p$ weights plays an important role in harmonic analysis and partial 
differential equations. 
For example, it is well known that $A_p$ weights can be characterized equivalently via the boundedness of
Hardy--Littlewood maximal functions and the Hilbert transform, the Riesz transforms in higher dimension.
Moreover, $A_p$ weights  also connect to the BMO space via the exponential and logarithm mapping, i.e., 
if $w$ is an $A_p$ weight, then 
$\log w$ is in $\bmo$, conversely, if $\log w\in\bmo$ then there is a 
$\gamma>0$ and $p>1$ such that $w^{\gamma}\in A_p$. 
%
%

In recent years, the sharp $A_p$ bound for Calder\'on--Zygmund operators has been obtained. 
The cases of the Hilbert and Riesz transforms were shown by Petermichl \cites{Pet2007,Pet2008},
the case of Haar shifts was proven by Lacey, Petermichl and Reguera \cite{LacPetReg2010},
for dyadic paraproducts by Beznosova \cite{Bez2008}, for the 
Bergman projection on the upper half plane by Pott--Reguera \cite{PotReg2013}  
and for general Calder\'on--Zygmund operators by Hyt\"onen \cite{Hyt2012}. 

%
%
%
Besides the $A_p$ class, in \cite{MucWhe1971}
Muckenhoupt and Wheeden also introduced the fractional weight class $A^\alpha_{p,q}$ in $\mathbb R^n$ as follows: a non-negative locally integrable function 
$w$ is in $ A^\alpha_{p,q}$ if
\begin{align*}
[w]_{A_{p,q}^\alpha}:=
\sup_{Q\textnormal{ a cube}}
  \left({1\over |Q|}\int_{Q}w(x)^q dx\right)\left({1\over |Q|}\int_{Q}w(x)^{-p'}dx\right)^{\frac{q}{p'}} <\infty,
\end{align*}
where $\frac{1}{p}-\frac{1}{q}=\frac{\alpha}{n}$.  When $\alpha=0$, then the class $A^\alpha_{p,q}$
becomes the classical $A_p$ weight.

They showed that 
$$\norm{I_\alpha:L^p(w^p)\to L^q(w^q)}<\infty$$ if and only if 
$[w]_{A_{p,q}^\alpha}<\infty$, where $I_\alpha$ is the standard fractional integral operators defined as
$$ I_\alpha f(x):=\int_{\R^n}f(y)\abs{x-y}^{\alpha - n}dy. $$

Later, a sharp version of this theorem was given 
by Lacey, Moen, P\'erez, and Torres \cite{LacMoePerTor2010} as follows.
\begin{TheoremA}[\cite{LacMoePerTor2010}]\label{T:osf}
Let $\frac{1}{p}-\frac{1}{q}=\frac{\alpha}{n}$ and let $w$ be in $A^\alpha_{p,q}$. There 
holds
\begin{align*}
\norm{I_\alpha: L^p(w^p)\to L^q(w^q)}
\lesssim [w]_{A_{p,q}^\alpha}^{(1-\frac{\alpha}{n})\max\{1,\frac{p'}{q}\}}
\end{align*}
and this result is sharp in the sense that there is a family of weights 
$\{w_\delta\}_{\delta\in\mathcal{A}}$ such that 
$$\norm{I_\alpha:L^p(w_\delta^p)\to L^q(w_\delta^q)}
\simeq [w_\delta]_{A_{p,q}^\alpha}^{(1-\frac{\alpha}{n})\max\{1,\frac{p'}{q}\}}.$$
\end{TheoremA}
They also showed the sharp weighted bound for the fractional maximal operator
(we remark that here and throughout the paper, for a measurable set $E$ 
we write $E(x)$ to mean the indicator function, i.e. $E(x)=\unit_{E}(x)$)
$$M_\alpha f(x):=\sup_{Q\textnormal{ a cube}}\frac{Q(x)}
{\abs{Q}^{1-\frac{\alpha}{n}}}\int_{Q}\abs{f(y)}dy.$$
\begin{TheoremB}[\cite{LacMoePerTor2010}]\label{T:osm}
Let $\frac{1}{p}-\frac{1}{q}=\frac{\alpha}{n}$ and let  $w$ be in $A^\alpha_{p,q}$. There 
holds
\begin{align*}
\norm{M_\alpha:L^p(w^p)\to L^q(w^q)}
\lesssim [w]_{A_{p,q}^\alpha}^{(1-\frac{\alpha}{n})\frac{p'}{q}}.
\end{align*}
\end{TheoremB}


%
%


It is well-known that the $A_p$ weights, Hilbert (Riesz) transforms,  $A_{p,q}^\alpha$ classes,  the fractional integral operators, and the corresponding quantitative estimates mentioned above
are associated with the standard Laplacian $\Delta$ in $\mathbb R^n$. Changing
the differential operator from the standard Laplacian $\Delta$ to other second order differential operators $L$ introduces new challenges and directions to explore, see for example some of the well-known results in the past 15 years \cite{KalVer1999,DM,CD,DY1,DY2,HMar,HMay,HMM}.

A natural question arises when changing the standard Laplacian $\Delta$ to another 
second order differential operator $L$: can we have new $A_p$ weights and 
$A_{p,q}^\alpha$ classes adapted to $L$ such that the related maximal functions, 
singular integrals and fractional integral operators have the right quantitative 
estimates in terms of the new $A_p$ or $A^\alpha_{p,q}$? 

In this paper, we focus on the Schr\"odinger operator  $L=-\Delta+V$ in 
$\mathbb{R}^n$,  $n\geq3$, where the non-negative function $V$ is in the reverse 
H\"older class. There has already been much work done on one--weight 
inequalities for these operators. However, there has never been sharp 
estimates (or any sort of quantitative estimates) for these operators. For 
the first time, we are able to prove such estimates. 

Quantitative bounds for the classical operators from harmonic analysis 
(e.g. Hilbert transform, Riesz transforms, maximal functions) are a deep 
reflection of the regularity of the classical Laplacian. Operators of the 
form $L=-\Delta + V$ present many challenges because they lack the regularity 
that $-\Delta$ possesses. In particular, the presence of the 
(non--negative) potential $V$ makes $L$ non--local in the sense that it is 
not invariant under translations and dilations. Of course, many techniques, 
theorems, and heuristics from classical harmonic analysis are based on the 
assumption that the operators under question possess this regularity that 
$L$ lacks.

%
%
%
%

In this Schr\"odinger setting,  a new class of $A_p$ weights associated to $L$ was introduced in \cite{BGS}, see also \cite{Tang2011}, which is a larger class, properly containing the classical Muckenhoupt $A_p$ weights. To be more precise, 
given $p>1$ we define $A_p^{\infty} 
= \cup_{\theta\geq0} A_p^{\theta}$, where $A_p^{\theta}$ is the set of
weights $w$ such that:
\begin{align}\label{weight}
[w]_{A_{p}^{\theta}}:=
\sup_{Q\textnormal{ a cube}}
\bigg({1\over {\psi}_\theta(Q) |Q| } \int_Q w(y)dy\bigg) 
\bigg( {1\over {\psi}_\theta(Q) |Q|} 
\int_Q w(y)^{-\frac{p'}{p}}  dy   \bigg)^{\frac{p}{p'}} < \infty,
\end{align}
where for each $\theta > 0$, $\psi_\theta$ on the collection 
of cubes $\{Q\}$ (with sides parallel to the coordinate axes) is defined by
\begin{align}\label{psi}
\psi_{\theta}(Q) := \Big(1+\frac{\ell (Q)}{\rho(Q)}\Big)^{\theta},
\end{align}
with $\rho(Q) := \rho(c_Q)$, $c_Q$ is the center of $Q$ and $\ell (Q)$ is the side-length of $Q$,  $\rho(x)$ is the critical function associated to the potential function $V$ (we refer to Section 2.2 for a precise definition).

We also have the fractional weight class $A^{\alpha,\theta}_{p,q}$ associated to $L$ defined as follows. 
Let $p>1$ and let $q$ be defined by 
$\frac{1}{p}-\frac{1}{q}=\frac{\alpha}{n}$. We define $A^{\alpha,\theta}_{p,q}$ as the class of weights $w$ such that:
\begin{align*}
[w]_{A^{\alpha,\theta}_{p,q}}
:=\sup_{Q\textnormal{ a cube}}
\left(\frac{1}{\psi_{\theta}(Q)\abs{Q}}
  \int_{Q}w^q(x)dx\right)
\left(\frac{1}{\psi_\theta(Q)\abs{Q}}
  \int_{Q}w^{-p'}(x)dx\right)^{\frac{q}{p'}}<\infty.
\end{align*}

In \cite{BGS} and \cite{Tang2011}, they showed that this new weight class $A_p^\infty$ satisfies most of the properties parallel to the classical Muckenhoupt $A_p$ weights, and they also established the weighted boundedness of  $M^\theta$,  the 
Hardy--Littlewood maximal function adapted to $L$ (we refer to Section 2 for the definition), and the Riesz transforms $\nabla L^{-1/2} $ in terms of  $A_p^\infty$, and the fractional integral operators $L^{-\alpha/2}$ in terms of $A^{\alpha,\theta}_{p,q}$.

We also note that the BMO space associated to $L$ was introduced in \cite{BGS1}, denoted by $\BMO_{\infty}$ (for a precise definition, we refer to Section 3). They also studied the boundedness of commutators of functions in $\BMO_{\infty}$ and the singular integrals adapted to $L$.  

In this paper, we aim to study the following results regarding the weights $A_p^\infty$ and  $A^{\alpha,\theta}_{p,q}$: 
\begin{enumerate}

\item the quantitative estimates for the Hardy--Littlewood maximal function associated to $L$ in terms of $A_p^\infty$; 

\item the quantitative estimates for the fractional integral operator associated to 
$L$, denoted by $L^{-\alpha/2}$, in terms of $A^{\alpha,\theta}_{p,q}$;

\item the ``exp-log'' link between $A_p^\infty$ and $\BMO_\infty$.

\end{enumerate}

To be more specific, the first main result of this paper 
%
consists of quantitative estimates for several versions of maximal functions associated to $L$. Here we mainly consider the Hardy--Littlewood type maximal function, the fractional maximal function, and the maximal function associated to the heat semigroup generated by $L$.
For $\theta>0$, and $0\leq\alpha<n$, the fractional maximal function $M^{\theta,\alpha}$ 
associated to $L$ is defined as:
\begin{align*}
M^{\theta,\alpha}f(x)
:=\sup_{Q}\frac{Q(x)}{\left(\pp{\theta}(Q)\abs{Q}\right)^{1-\frac{\alpha}{n}}}
  \int_{Q}\abs{f(y)}dy.
\end{align*}
In particular, when $\alpha=0$, we denote $M^\theta f(x):=M^{\theta,0} f(x)$, which is the Hardy--Littlewood type maximal function associated to $L$. We also recall the heat maximal function $M^L$ associated to $L$:
\begin{align}\label{D:heatmax}
M^{L}f(x) 
:= \sup_{t\geq 0} |e^{-tL}f(x)|.
\end{align}

%
%
%


Then we have the following quantitative estimates for
the Hardy--Littlewood type maximal function associated to $L$ and the maximal heat semigroup.
\begin{theorem}\label{T:sharp2}Suppose $\theta> 0$. Then we have that
\begin{itemize}
 \item[(1)]$w(\{x\in\mathbb R^n:\ M^\theta f(x)> \lambda\})
 \leq [w]_{A_p^{\theta}}\Big(\frac{\norm{f}_{L^p(w)}}{\lambda}\Big)^{p}$ for all $\lambda>0$ and for every $f\in L^p(\mathbb R^n)$ with $1< p<\infty$;
 \item [(2)] There is a $C_\theta$ so that
 \begin{align*}
  \norm{M^{L}:L^p(w) \to L^p(w)}
  \leq C_\theta \norm{M^{\theta}:L^p(w)\to L^p(w)}.
 \end{align*}
 As a consequence, we see that $M^L$ possesses the same quantitative estimate $M^\theta$ does.
\end{itemize}

\end{theorem}

Moreover, we also have the following results regarding 
  the fractional maximal function $M^{\theta,\alpha}$ 
associated $L$.
\begin{theorem}\label{T:sharp}
Suppose $0\leq\alpha<n$, $\frac{1}{p}-\frac{1}{q}
=\frac{\alpha}{n}$. Let $\gamma = \theta/(1+\frac{p'}{q})$, then
 \begin{align*}
 \norm{M^{\theta,\alpha}:L^p(w^p)\to L^q(w^q)}
 \leq 
  C [w]_{{A}_{p,q}^{\alpha,\frac{\gamma}{3}}}^{\frac{p'}{q}(1-\frac{\alpha}{n})}.
 \end{align*}
\end{theorem}

The third main result of this paper is a quantitative estimate of  the fractional integral 
operator $L^{-\frac{\alpha}{2}}f(x)$.

\begin{theorem}\label{T:mainfio}
Suppose $0\leq\alpha<n$. Let $1<p<\frac{n}{\alpha}$ and $q$ be defined by the equation 
$\frac{1}{q}=\frac{1}{p}-\frac{\alpha}{n}$ and let 
$K$ be defined by the equation $\left(\frac{1}{K} +
\frac{1}{K}\frac{q}{p'}\right)(1-\frac{\alpha}{n})
  \max\{1,\frac{p'}{q}\}
=\frac{1}{2}$. For $w\in {A}_{p,q}^{\alpha,\theta / 3K}$ there holds
\begin{align*}
\norm{L^{-\frac{\alpha}{2}}:\, L^p(w^p)\to L^q(w^q)}
\lesssim 
  [w]_{{A}_{p,q}^{\alpha,\theta / 3K}}^{(1-\frac{\alpha}{n})
  \max\{1,\frac{p'}{q}\}},
\end{align*}
where the implied constant depends on $p,q,\alpha,n$, and $\theta$.
\end{theorem}

Here we point out that the maximal operator associated to the heat semigroup $M^L$ and the Hardy--Littlewood type maximal function
$M^\theta$ satisfy the quantitative estimate as in Theorem B for $\alpha=0$, and that  the fractional maximal function $M^{\theta,\alpha}$ satisfies the quantitative estimate as in Theorem B. Moreover, the fractional integral 
operator $L^{-\frac{\alpha}{2}}f(x)$ satisfies the quantitative estimate as in Theorem A.

However, we remark that the class of weights $A_{p}^{\infty}$ (resp. $A^{\alpha,\theta}_{p,q}$) is associated to $L$, and can be \textit{much}
larger than the standard $A_p$ (resp. $A^{\alpha}_{p,q}$) classes. 
Typical examples are as follows.
\begin{example}\label{T:example}
Consider $L:=-\Delta + 1$ on $\mathbb R^n$. Then we have that $\rho(x)\equiv1$. 
Then consider the function $w(x):=1+|x|^\gamma$ with $\gamma> n(p-1)$.  
We see that $w$ is in $A_p^\infty$, however, $w$ is not in classical $A_p$.
\end{example}

%

We remark that it might be more precise to decorate the various operators, 
weight classes and other objects we define in this paper with the letter 
``$L$'', but to avoid cumbersome notation, we do not do this.   From the context at hand it should be clear.

In the end, we have a study on  the ``exp-log'' link of $A_p^\infty$
and $\BMO_\infty$. To be more specific, we show that
\begin{theorem}\label{T:bmoas}
(i) If $w\in A_{p}^{\infty}$, then we have $\log w\in \BMO_{\infty}$;

(ii) However, the converse is not true in general.
\end{theorem}

The outline  and structure of the paper is as follows.     In Section \ref{s:2} we recall some fundamental facts
for Schr\"odinger operators with non-negative potential $V$.

In Section \ref{s:mo}, we will 
develop some of the weighted theory associated to the classes 
$A_{p}^\theta$ and $A_{p,q}^{\alpha,\theta}$. We will discuss the 
operators $M^\theta$ and $M^{\theta,\alpha}$ in more detail and prove 
Theorems  \ref{T:sharp2} and \ref{T:sharp}.
A key feature in this section 
is the introduction of a slightly different critical function 
that we denote $\widetilde{\rho}$. There is also the corresponding $\bp{\theta}$
function and $\widetilde{A}_{p,q}^{\alpha,\theta}$ classes. These new 
functions are much less sensitive to the precise location of the cube at 
which they are evaluated. In particular, if $Q\subset Q'$ then there holds 
$\bp{\theta}(Q)^{-1}\leq\bp{\theta}(Q')^{-1}$. This is an important modification 
as it mitigates the non--locality of the Schr\"odinger operator.

In Section \ref{s:fio} we will prove Theorem \ref{T:mainfio}. For this 
section, we will show that $L^{-\frac{\alpha}{2}}$ is dominated by an appropriate 
dyadic operator. An important step in this procedure is organizing the cubes 
in to sub--collections on which $\bp{\theta}(Q)$ is roughly equal to $2^r$ for 
$r\in\mathbb{N}$. This further mitigates the non--locality of the 
Schr\"odinger operator as it allows us to essentially ignore the $\bp{\theta}$ 
function for most of the argument. 

In Section \ref{s:wtdthry}, we recall the definition of BMO spaces associated to $L$ and the related properties. And then we will prove Theorem 
\ref{T:bmoas}. The main technique here for (i) are Jensen inequalities and is similar to the classical 
case. We also point out that in general, the reverse direction ``exp'' is not true.

Finally, in Section \ref{s:con} we give some concluding remarks. In particular, we 
prove some new two weight inequalities for $L^{-\frac{\alpha}{2}}$. We also give some 
potential areas of investigation.

\section{Preliminaries}\label{s:2}
\setcounter{equation}{0}
In this section we set some notation and recall the well-known facts and results 
related to Schr\"odinger operator $L=-\Delta+V$ on $\mathbb R^n$ for $n\geq 3$.

We first recall that for a subset $E$ we will write $E(x)$ for the indicator 
function of $E$; that is $E(x):=\unit_{E}(x)$. If $Q$ is a cube, then 
$\ell (Q)$ will denote the side-length of $Q$. 
\subsection{Reverse H\"older class}

We say that the function  $V$ satisfies a Reverse H\"older property of 
order $\sigma > n/2$ and write $V\in\rh_\sigma$, if  
there exists a positive constant $C$ such that for all cubes $Q$ there holds 
\begin{align}\label{RH}
\left(\frac{1}{\abs{Q}}\int_{Q}V(y)^\sigma dy\right)^{\frac{1}{\sigma}}
\leq \frac{C}{\abs{Q}}\int_{Q}V(y)dy.
\end{align}
For $\sigma=\infty$, the left hand side of \eqref{RH} is replaced by the essential supremum over $B$. It is well-known that elements of $\RH_\sigma$ are doubling measures, and that $\RH_\sigma\subset \RH_{\sigma'}$ whenever $\sigma'<\sigma$ .

\subsection{The critical function $\rho(x)$}

 Associated to 
$V$ we have the critical function $\rho$ introduced in \cite{Sh}, defined by 
\begin{align}\label{rho}
\rho(x):= \Big( \sup\Big\{ r>0:\ {1\over r^{n-2}}
\int_{B(x,r)} V(y)dy \leq 1  \Big\} \Big)^{-1}.
\end{align}
As an example for the harmonic oscillator with $V(x)=|x|^2$, we have $\rho(x)\sim (1+|x|)^{-1}$.

We state the following property of $\rho$; for the proof see \cite{Sh}.
\begin{lemma}\label{Lem1: rho}
Let $\rho$ be the critical radius function associated with $L$ defined in \eqref{rho}. 
\begin{enumerate}[{\upshape (i)}]
\item There exist positive constants $k_0 \ge 1$ and $C_0>0$ so that
$$
 {\rho(x)\over C_0 [\rho(x)+|x-y|]^{k_0}}\leq \rho(y)\leq C_0 \rho(x)[\rho(x)+|x-y|]^{k_0/(1+k_0)},
$$
for all $x,y\in \mathbb{R}^n$.  In particular, for any ball $B\subset \mathbb R^n$, and any $x,y\in B$, we have $\rho(x)\le C_0^2 \bigl(1+\f{r_B}{\rho_B}\bigr)^2\rho(y)$.
 \item There exists $C>0$ and $\sigma_0=\sigma_0(\sigma,n)$ so that
 $$
 \f{1}{r^{n-2}}\int_{B(x,r)}V(y)dy\leq C\Big(\f{r}{R}\Big)^{\sigma_0}\f{1}{R^{n-2}}\int_{B(x,R)}V(y)dy
 $$
 for all $x\in \mathbb{R}^n$ and $R>r>0$. 
 \item For any $x\in \mathbb{R}^n$, we have
 $$
 \f{1}{\rho(x)^{n-2}}\int_{B(x,\rho(x))}V(y)dy=1.
 $$
 \item There exists $C>0$ so that for any $r>\rho(x)$ 
 $$ r^2 \barint_{B(x,\rho(x))} V(y)\,dy \le C \Big(\f{r}{\rho(x)}\Big)^{n_0-n+2}
 $$
 where $n_0$ is the doubling order of $V$. That is, $\int_{2B} V\lesssim 2^{n_0} \int_B V$ for any ball $B$. 
\end{enumerate}
\end{lemma}
 \begin{remark}\label{rem: V estimate}
 It follows from Lemma \ref{Lem1: rho} (ii) and (iii) that for any ball $B$,
 \begin{align*}
 	r_B^2\barint_{B}V(y)\,dy \lesssim \left\lbrace\begin{array}{ll}
 		\Bigl(\dfrac{r_B}{\rho_B}\Bigr)^{\sigma_0} \qquad & r_B\le \rho_B,\\[10pt]
 		\Bigl(\dfrac{r_B}{\rho_B}\Bigr)^{n_0+2-n} \qquad & r_B>\rho_B.
 		\end{array}
 		\right. 
 	\end{align*}
 \end{remark}

%

\subsection{Heat kernel bounds for $L$}

We now recall the heat kernel upper bounds for the Schr\"odinger 
operator.

Denote by $p_{t,L}(x,y)$ the 
integral kernel of the semigroup $\{e^{-tL}\}_{t>0}$ generated by $-L=\Delta-V$ and 
by $p_{t}(x,y)$ the integral kernel of the semigroup $e^{-t\Delta}$ generated by 
$\Delta$. Then obviously we have
$$ 0\leq p_{t,L}(x,y) \leq p_t(x,y) := (4\pi t)^{-n/2} \exp(-|x-y|^2/4t). $$

%
%
%

We recall the well-known heat kernel upper bounds for the Schr\"odinger operator as well as properties for $V$ and its critical radius function $\rho$ as defined in \eqref{rho}.  The following estimates on the heat kernel of $L$ are well-known.
\begin{proposition}[\cites{DZ2,DZ3}]\label{prop-kernelestimates}
	Let $L=-\Delta+V$ with $V\in RH_\sigma$ for some $\sigma\ge n/2$. Then for each $N>0$ there exists $C_N>0$ and $c>0$ such that
	\begin{align}\label{hk bound}
		p_{t,L}(x,y)\le C_N\frac{e^{-|x-y|^2/ct}}{t^{n/2}} \ContainC{1+\frac{\sqrt{t}}{\rho(x)}+\frac{\sqrt{t}}{\rho(y)}}^{-N}
	\end{align}
	and
	\begin{align}\label{hk holder}
		|p_{t,L}(x,y)-p_{t,L}(x',y)|\le C_N\ContainC{\frac{|x-x'|}{\sqrt{t}}}^{\sigma_1} \frac{e^{-|x-y|^2/ct}}{t^{n/2}} \ContainC{1+\frac{\sqrt{t}}{\rho(x)}+\frac{\sqrt{t}}{\rho(y)}}^{-N}
	\end{align}
	whenever $|x-x'|\le \sqrt{t}$ and for any $0<\sigma_1<\sigma$.
\end{proposition}

%
%
%

\section{Maximal Operators Associated to $L$}\label{s:mo}
In this section, we define some maximal operators associated to $L$ and 
we give quantitative bounds for their norms as operators acting on $L^p(w)$. 
The bounds will be in terms of the classes of weights defined and discussed 
in Section \ref{s:wtdthry}. As above, in the case $V\equiv 0$, all of these 
operators will reduce to the classical Hardy--Littlewood maximal function.

\subsection{New classes of weights associated to $L$}


To define the next two classes of weights, we will make use of the critical radius 
function $\rho$. Let $Q$ be a cube and define the following functions:
\begin{align}\label{new rho}
\widetilde{\rho}(Q) := \sup_{x\in Q}\rho(x)
\quad{\rm and}\quad
\widetilde{\psi}_{\theta}(Q) 
:= \left(1+ \frac{\ell (Q)}{\widetilde{\rho}(Q)}\right)^{\theta}.
\end{align}

Using the function $\widetilde{\psi}_{\theta}$ introduced above, we define 
$\widetilde{A}_{p}^{\infty}
:=\cup_{\theta\geq 0}\widetilde{A}_{p}^{\theta}$, where 
$\widetilde{A}_{p}^{\theta}$ is the set of weights $w$ such that:
\begin{align}\label{bigweight}
[w]_{\widetilde{A}_{p}^{\theta}}:=
\sup_{Q\textnormal{ a cube}}
\bigg({1\over \widetilde{\psi}_\theta(Q) |Q| } \int_Q w(y)dy\bigg) 
\bigg( {1\over \widetilde{\psi}_\theta(Q) |Q|} 
\int_Q \sigma(y)  dy   \bigg)^{\frac{p}{p'}} < \infty.
\end{align}

The classes $A_{p}^{\infty}$ and $A_{p}^{\theta}$ were introduced 
in studied further in (for example) \cites{Tang2011,BGS}. The classes 
$\widetilde{A}_{p}^{\infty}$ and 
$\widetilde{A}_{p}^{\theta}$ are -- to our knowledge -- new. We 
need these classes because the standard classes are are very 
``non--local'' in the sense that the functions $\psi_{\theta}$ depend 
on the precise location of $Q$. The functions $\widetilde{\psi}_\theta$ are not 
as sensitive to the precise location of $Q$. The classes $A_{p}^{\theta}$ 
and $\widetilde{A}_{p}^{\theta}$ are related by the following proposition.

\begin{proposition}\label{L:bigwtsmlwt}
For all $\theta \geq 0$ and weights $w$ there holds
\begin{align}\label{E:bigwtsmlwtest}
[w]_{\widetilde{A}_{p}^{3\theta}}
\simeq [w]_{{A}_{p}^{\theta}}, 
\end{align}
and $A_{p}^{\infty}=\widetilde{A}_{p}^{\infty}$. 
\end{proposition}
\begin{proof}
Clearly $[w]_{{A}_{p}^{\theta}}\leq [w]_{\widetilde{A}_{p}^{\theta}}$
and so to prove \eqref{E:bigwtsmlwtest}, it suffices to show that
\begin{align}\label{E:bigwtsmlwtest eee}
[w]_{\widetilde{A}_{p}^{\theta}} \lesssim [w]_{{A}_{p}^{\theta}}.
\end{align} 
Let $Q$ be a cube and let $x\in Q$. Using Lemma \ref{Lem1: rho} we have 
$\rho(x)\lesssim\left(1+\frac{\ell (Q)}{\rho(x)}\right)^2\rho(c_Q)$ and so
\begin{align*}
\psi_\theta(Q)
=\left(1+\frac{\ell (Q)}{\rho(c_Q)}\right)^{\theta}
\lesssim\left(1 + \frac{\ell (Q)}{\rho(x)}\left(1+\frac{\ell (Q)}{\rho(x)}\right)^2\right)^{\theta}
\lesssim \left(1+ \frac{\ell (Q)}{\rho(x)}\right)^{3\theta}.
\end{align*}
Thus $1 /\widetilde{\psi}_{3\theta}(Q)\lesssim 1 /\psi_{\theta}(Q)$ and so 
\eqref{E:bigwtsmlwtest eee} holds.
\end{proof}

We also have a new generalization of the $A_{p,q}$ classes of Muckenhoupt 
and Wheeden \cite{MucWhe1971}. Using the new auxiliary function introduced in \eqref{new rho}, define the $\widetilde{A}_{p,q}^{\alpha,\theta}$ characteristic of a 
weight $w$ by:
\begin{align*}
[w]_{\widetilde{A}_{p,q}^{\alpha,\theta}}
:=\sup_{Q\textnormal{ a cube}}
\left(\frac{1}{\widetilde{\psi}_{\theta}(Q)\abs{Q}}
  \int_{Q}w^q(x)dx\right)
\left(\frac{1}{\widetilde{\psi}_\theta(Q)\abs{Q}}
  \int_{Q}w^{-p'}(x)dx\right)^{\frac{q}{p'}}.
\end{align*}

Similarly, using the new auxiliary function introduced in \eqref{new rho}, for $\theta>0$, and $0\leq\alpha<n$ we define the maximal function $\widetilde{M}^{\theta,\alpha}$ 
associated $L$:
\begin{align*}
\widetilde{M}^{\theta,\alpha}f(x)
:=\sup_{Q}\frac{Q(x)}{\left(\widetilde{\psi}_{\theta}(Q)\abs{Q}\right)^{1-\frac{\alpha}{n}}}
  \int_{Q}\abs{f(y)}dy,
\end{align*}
and in particular, when $\alpha=0$, we denote
$$\hspace{.25in}
\widetilde M^\theta f(x):=\widetilde M^{\theta,0} f(x).
$$


\subsection{Quantitative Bounds} 
In this section, we will give quantitative bounds for the maximal operators 
defined above. Ideally, we would like to give quantitative bounds for $M^\theta$ 
in terms of the $A_{p}^\theta$ characteristic of the weight, and we would like to 
prove similar assertions for the other maximal operators defined. However, for 
some of the operators it seems the bounds must be given in terms of the 
$A_{p}^\gamma$ characteristic, where $\gamma < \theta$. This is also 
true in the qualitative versions of these theorems in \cites{Tang2011,BGS}.
However, we are able to give the desired quantitative weak bounds. 

The proof of Theorem  \ref{T:sharp2} will be by the following two lemmas. The first lemma provides the 
weak-type quantitative estimates of $M^\theta$ as required, the proof of which follows from   the 
standard  Besicovitch covering lemma. As a consequence,  the estimate in $(1)$ of Theorem  \ref{T:sharp2} will be proven.

The second
lemma will establish a pointwise bound that easily implies  
the estimate in $(2)$ of Theorem  \ref{T:sharp2}, the proof of which follows from the pointwise upper bound of the heat kernel.

\begin{lemma}\label{L:weakp}
For $1<p<\infty$, there holds:
\begin{align*}
w\left(\{M^{\theta} f > \lambda\}\right)
\lesssim [w]_{A_{p}^{\theta}}\left(\frac{\norm{f}_{L^p(w)}}{\lambda}\right)^p.
\end{align*}
\end{lemma}
\begin{proof}
Let $\Omega_{\lambda}=\{M^\theta f> \lambda\}$ and let $K_\lambda$ be any
compact subset of $\Omega_\lambda$. For every $x\in K_\lambda$ there is a 
cube $Q_x$ containing $x$, such that:
\begin{align*}
\frac{1}{\psi_\theta(Q_x)\abs{Q_x}}\int_{Q_x}\abs{f(y)}dy
>\frac{\lambda}{2}.
\end{align*}
Since this set is compact, by the Besicovitch covering lemma, there is a 
number $M=M(n)$ such that there are $M$ collections of sets 
$\mathcal{Q}_{1},\ldots,\mathcal{Q}_M$ such that each $\mathcal{Q}_j
=\{Q_y:y\in K_\lambda\}$ and the sets in each $\mathcal{Q}_j$ are pairwise 
disjoint. Additionally, $K_\lambda \subset \cup_{j=1}^{M}\mathcal{Q}_j$. 
In other words, $K_\lambda$ is covered by $M$ collections of disjoint cubes. 
Thus it is enough to fix a $1\leq j \leq M$ and set $\mathcal{Q}=\mathcal{Q}_j$
and estimate $\sum_{Q\in\mathcal{Q}}w(Q)$. 

Note that for a $Q\in\mathcal{Q}$ there holds:
\begin{align*}
w(Q)
\leq 2\int_{\R^n} \frac{w(Q)}{{\psi}_{\theta}(Q)\abs{Q}\lambda}f(x)Q(x)dx.
\end{align*}
Using this we have:
\begin{align*}
\sum_{Q\in\mathcal{Q}}w(Q)
&\leq\int_Q \sum_{Q\in\mathcal{Q}}\frac{w(Q)}
  {{\psi}_{\theta}(Q)\abs{Q}\lambda}
  f(x)\sigma(x)^{1/p'}w(x)^{1/p}dx
\\
&\leq \left(\int\abs{\sum_{Q\in\mathcal{Q}}
  \frac{w(Q)}
  {{\psi}_{\theta}(Q)\abs{Q}\lambda}Q(x)}^{p'}\sigma(x)dx\right)^{\frac{1}{p'}}\cdot
  \norm{f}_{L^p(w)}.
\end{align*}
Now, the cubes $Q\in\mathcal{Q}$ are maximal and are thus disjoint. So 
the first term in the right-hand side of the last inequality above is equal to:
\begin{align*}
\frac{1}{\lambda}
  \left(\sum_{Q\in\mathcal{Q}}\frac{w(Q)^{p'-1}\sigma(Q)}
  {{\psi}_{\theta}(Q)^{p'}\abs{Q}^{p'}}w(Q)\right)^{\frac{1}{p'}}
&\leq \frac{[w]_{A_{p}^{\theta}}^{1/p}}{\lambda}
  \left(\sum_{Q\in\mathcal{Q}}w(Q)\right)^{1/p'}.
\end{align*}
Thus there holds:
\begin{align*}
\sum_{Q\in\mathcal{Q}}w(Q)
\leq [w]_{A_{p}^{\theta}}^{1/p}
  \left(\sum_{Q\in\mathcal{Q}}w(Q)\right)^{\frac{1}{p'}}
  \frac{\norm{f}_{L^p(w)}}{\lambda},
\end{align*}
which is the required estimate.
\end{proof}

\begin{lemma}\label{L:heattopsi}
For $\theta\in(0, \infty)$, there exists a constant $C_\theta$ such that for any 
locally integrable function $f$, and for every $x\in\mathbb{R}^n$ and $t>0$, we have
\begin{align}
|e^{-tL}f(x)|\leq C_\theta M^\theta(f)(x).
\end{align}
\end{lemma}
\begin{proof}
For any fixed $x\in\mathbb{R}^n$ and $t>0$,
\begin{align*}
|e^{-tL}f(x)|&\leq \int_{\mathbb{R}^n} p_t(x,y)|f(y)|dy\\
&\leq  C_\theta \int_{\mathbb{R}^n} \frac{e^{-|x-y|^2/ct}}{t^{n/2}} \ContainC{1+\frac{\sqrt{t}}{\rho(x)}+\frac{\sqrt{t}}{\rho(y)}}^{-\theta}|f(y)|dy,
\end{align*}
where $\theta$ is any positive constant.

We now denote by $B:= B(x,\sqrt{t})$ the ball in $\mathbb{R}^n$ centered at $x$ with radius $\sqrt t$. Then we have
\begin{align*}
|e^{-tL}f(x)|&\leq \int_{\mathbb{R}^n} p_t(x,y)|f(y)|dy\\
&\leq  C_\theta \sum_{j=1}^\infty \int_{2^jB\backslash 2^{j-1}B} \frac{e^{-|x-y|^2/ct}}{t^{n/2}} \ContainC{1+\frac{\sqrt{t}}{\rho(x)}+\frac{\sqrt{t}}{\rho(y)}}^{-\theta}|f(y)|dy\\
&\quad + C_\theta  \int_{B} \frac{e^{-|x-y|^2/ct}}{t^{n/2}} \ContainC{1+\frac{\sqrt{t}}{\rho(x)}+\frac{\sqrt{t}}{\rho(y)}}^{-\theta}|f(y)|dy\\
&\leq  C_\theta\sum_{j=1}^\infty 2^{j\theta} e^{-c2^{2(j-1)}} {1\over |B|}\int_{2^jB\backslash 2^{j-1}B}  
\ContainC{1+\frac{2^j\sqrt{t}}{\rho(x)}}^{-\theta}|f(y)|dy\\
&\quad + C_\theta  {1\over \widetilde{\psi}_N(B) |B|}\int_{B}  |f(y)|dy\\
&\leq  C_\theta\sum_{j=1}^\infty 2^{j(\theta+n)} e^{-c2^{2(j-1)}} {1\over \widetilde{\psi}_\theta(2^jB) |2^jB|}\int_{2^jB\backslash 2^{j-1}B}  
|f(y)|dy\\
&\quad + C_\theta M^\theta(f)(x) \\
&\leq C_\theta M^\theta(f)(x).
\end{align*}
\end{proof}


To prove Theorem \ref{T:sharp}, we first note that
$$\norm{M^{\theta,\alpha}:L^p(w^p)\to L^q(w^q)}
 \leq\norm{\widetilde{M}^{\theta,\alpha}:L^p(w^p)\to L^q(w^q)}$$
follows easily from the definitions of
$M^{\theta,\alpha}$ and $\widetilde{M}^{\theta,\alpha}$. Hence,
only 
the estimate 
$$\norm{\widetilde{M}^{\theta,\alpha}:L^p(w^p)\to L^q(w^q)}
 \lesssim [w]_{\widetilde{A}_{p,q}^{\alpha,\gamma}}^{\frac{p'}{q}(1-\frac{\alpha}{n})}$$
needs to be shown.   And then, Theorem \ref{T:sharp} follows from the above quantitative estimates and from
 Proposition
\ref{L:bigwtsmlwt}, which shows that
$[w]_{\widetilde{A}_{p,q}^{\alpha,\gamma}}\simeq [w]_{{A}_{p,q}^{\alpha,\frac{\gamma}{3}}}$.

To begin with, we need the following two universal bounds for 
weighted maximal functions. Let $\mu$ be a weight and $0\leq\alpha<n$. Define
\begin{align*}
M_{\mu}^{\alpha}f(x)
:=\sup_{Q}\frac{Q(x)}{\mu(Q)^{1-\frac{\alpha}{n}}}\int_{Q}\abs{f(y)}\mu(y)dy
\hspace{.25in}
M_{\mu}f(x):=M_{\mu}^{0}f(x).
\end{align*}
There holds
\begin{lemma}\label{L:univwtfrac}
Let $\mu$ be a weight and let $0\leq\alpha<n$ and $\frac{1}{p}-\frac{1}{q}=\frac{\alpha}{n}$. 
Then $$\norm{M_{\mu}^{\alpha}:L^p(\mu)\to L^q(\mu)}\lesssim 1,$$ where the implied 
constant does not depend on $\mu$. 
\end{lemma}
When $\alpha = 0$, this is the well--known Doob maximal inequality. For 
$0<\alpha<n$ see \cite{LacMoePerTor2010}*{Lemma 4.1}. We will use these facts to 
prove the following theorem. 
\begin{theorem}
Let $0\leq\alpha < n$ and $\frac{1}{p}-\frac{1}{q}=\frac{\alpha}{n}$. There holds
\begin{align*}
\norm{\widetilde{M}^{\theta,\alpha}:L^p(w^p)\to L^q(w^q)}
\lesssim [w]_{\widetilde{A}_{p,q}^{\alpha,\gamma}}^{\frac{p'}{q}(1-\frac{\alpha}{n})}.
\end{align*}
\end{theorem}
\begin{proof}
The proof of the following theorem follows the corresponding proof 
in \cite{LacMoePerTor2010}.
Let $u=w^q$ and $\sigma=w^{-p'}$ and $r=1+\frac{q}{p'}$. There holds 
$\frac{p'}{q}(1-\frac{\alpha}{n})=\frac{r'}{q}$. For any 
cube $Q$ we have 
\begin{align*}
&\frac{1}{\left(\widetilde{\psi}_\theta(Q)\abs{Q}\right)^{1-\frac{\alpha}{n}}}
  \int_{Q}\abs{f(y)}dy\\
&=\left(\frac{\sigma(Q)u(Q)^{\frac{p'}{q}}}
  {\widetilde{\psi}_\theta(Q)\abs{Q}^{1+\frac{p'}{q}}}\right)^{1-\frac{\alpha}{n}}
  \left(\frac{\abs{Q}}{u(Q)}\right)^{\frac{p'}{q}(1-\frac{\alpha}{n})}
  \frac{1}{\sigma(Q)^{1-\frac{\alpha}{n}}}\int_{Q}\abs{f(y)}dy.
\end{align*}
Let $\gamma$ satisfy $\gamma\frac{p'}{q}+\gamma=\theta$ (i.e. $\gamma=\frac{\theta}
{1+p'/q}$). Then this becomes
\begin{align*}
\left\{
\left(\frac{u(Q)}{\psi_\gamma(Q)\abs{Q}}\right)^{\frac{p'}{q}}
\left(\frac{\sigma(Q)}{\widetilde{\psi}_\gamma(Q)\abs{Q}}\right)
\right\}^{1-\frac{\alpha}{n}}
\left(\frac{\abs{Q}}{u(Q)}\right)^{\frac{p'}{q}(1-\frac{\alpha}{n})}
  \frac{1}{\sigma(Q)^{1-\frac{\alpha}{n}}}\int_{Q}\abs{f(y)}dy.
\end{align*}
Estimating the first factor from above by 
$[w]_{\widetilde{A}_{p,q}^{\alpha,\theta}}^{\frac{p'}{q}(1-\frac{\alpha}{n})}$, this is dominated by
\begin{align*}
[w]_{\widetilde{A}_{p,q}^{\alpha,\gamma}}^{\frac{p'}{q}(1-\frac{\alpha}{n})}
  \left(
  \frac{\abs{Q}}{u(Q)}
  \right)^{\frac{p'}{q}(1-\frac{\alpha}{n})}
  \frac{1}{\sigma(Q)^{1-\frac{\alpha}{n}}}\int_{Q}
  \frac{\abs{f(y)}}{\sigma(y)}\sigma(y)dy.
\end{align*}
Applying H\"older's Inequality with exponents $q/r'$ and $(q/r')'= (1-\frac{r'}{q})^{-1}$ 
and noting that $1-\frac{r'}{q}=1-\frac{p'}{q}(1-\frac{\alpha}{n})$, this last expression is then 
dominated by
\begin{align*}
[w]_{\widetilde{A}_{p,q}^{\alpha,\gamma}}^{\frac{p'}{q}(1-\frac{\alpha}{n})}
\left(\frac{1}{u(Q)}\int_{Q}M_{\sigma}^{\alpha}(f\sigma^{-1})^{q/r'}dx\right)^{r'/q}.
\end{align*}
Taking a supremum over all cubes centered at $x$ we have the pointwise inequality
\begin{align*}
\widetilde{M}^{\theta,\alpha}f(x)
\leq [w]_{\widetilde{A}_{p,q}^{\alpha,\gamma}}^{\frac{p'}{q}(1-\frac{\alpha}{n})}
  M_{u}\left\{
  M_{\alpha,\sigma}(f\sigma^{-1})^{q/r'}u^{-1}\right\}(x)^{r'/q}.
\end{align*}
Thus, we have
\begin{align*}
\norm{\widetilde{M}^{\theta,\alpha}f}_{L^q(w^q)}
&=\norm{\widetilde{M}^{\theta,\alpha}f}_{L^q(u)}
\\&\leq[w]_{\widetilde{A}_{p,q}^{\alpha,\gamma}}^{\frac{p'}{q}(1-\frac{\alpha}{n})}
  \norm{
  M_{u}\left\{
  M_{\alpha,\sigma}(f\sigma^{-1})^{q/r'}u^{-1}\right\}
  }_{L^{r'}(u)}^{\frac{r'}{q}}
\\&\lesssim[w]_{\widetilde{A}_{p,q}^{\alpha,\gamma}}^{\frac{p'}{q}(1-\frac{\alpha}{n})}
  \norm{
  M_{\alpha,\sigma}(f\sigma^{-1})^{q/r'}u^{-1}
  }_{L^{r'}(u)}^{\frac{r'}{q}}.
\end{align*}
Let $s=\frac{pr'}{q}$. Then $\frac{1}{s}-\frac{1}{r}=\frac{\alpha}{n}$ and so 
by Lemma \ref{L:univwtfrac} there holds
\begin{align*}
\norm{M_{\alpha,\sigma}(f\sigma^{-1})^{q/r'}u^{-1}}_{L^{r'}(u)}^{\frac{r'}{q}}
&=\norm{M_{\alpha,\sigma}(f\sigma^{-1})^{q/r'}}_{L^{r'}(u^{-\frac{r'}{r}})}^{\frac{r'}{q}}
\lesssim\norm{(f\sigma^{-1})^{q/r'}}_{L^{s}(u^{-\frac{r'}{r}})}^{\frac{r'}{q}}.
\end{align*}
Observe that $\sigma^{-s\frac{q}{r'}}u^{-\frac{r'}{r}}=w^p$ (do this by writing 
$\sigma$ and $u$ in terms of $w$ and then do some gymnastics with the 
H\"older exponents), and clearly $\abs{f}^{s\frac{q}{r'}}=\abs{f}^{p}$ and so 
this last line is equal to $\norm{f}_{L^p(w^p)}$ as desired. 
\end{proof}

\section{Fractional Integral Operator}\label{s:fio}
The goal of this section is to prove Theorem \ref{T:mainfio}. We first 
recall some definitions. The heat semigroup associated to  
$L$ is a family of operators given by $H_tf(x):=e^{-tL}f(x)$. For $0<\alpha < n$ 
using the functional calculus we can write $L^{-\frac{\alpha}{2}}$ as an integral 
operator:
\begin{align*}
L^{-\frac{\alpha}{2}}f(x)
=\int_{0}^{\infty}e^{-tL}f(x)t^{\alpha /2 - 1}dt.
\end{align*} 
We will prove a quantitative version of a theorem of Tang \cite{Tang2011}. 
This is a version of the theorem of Lacey--Moen--Torres--P\'erez 
adapted to our setting \cite{LacMoePerTor2010}.
\begin{theorem}
Let $1<p<\frac{n}{\alpha}$ and $q$ be defined by the equation 
$\frac{1}{q}=\frac{1}{p}-\frac{\alpha}{n}$ and let 
$K$ be defined by the equation $\left(\frac{1}{K} +
\frac{1}{K}\frac{q}{p'}\right)(1-\frac{\alpha}{n})
  \max\{1,\frac{p'}{q}\}
=\frac{1}{2}$. For $w\in {A}_{p,q}^{\alpha,\theta / 3K}$ there holds
\begin{align*}
\norm{L^{-\frac{\alpha}{2}}:L^p(w^p)\to L^q(w^q)}
\lesssim 
  [w]_{{A}_{p,q}^{\alpha,\theta / 3K}}^{(1-\frac{\alpha}{n})
  \max\{1,\frac{p'}{q}\}},
\end{align*}
where the implied constant depends on $p,q,\alpha,n$, and $\theta$.
\end{theorem}

Recalling Lemma \ref{L:bigwtsmlwt}, Theorem \ref{T:mainfio} will follow 
from the following lemma.
\begin{lemma}\label{L:mainfio}
Let $1<p<\frac{n}{\al}$ and $q$ be defined by the equation 
$\frac{1}{q}=\frac{1}{p}-\frac{\al}{n}$ and let $w$ and $\sigma$ 
be weights. There holds
\begin{align}\label{E:mainfio}
\norm{L^{-\frac{\alpha}{2}}:L^p(w^p)\to L^q(w^q)}
\lesssim [w]_{\widetilde{A}_{p,q}^{\alpha,\theta / K}}^{(1-\frac{\alpha}{n})
  \max\{1,\frac{p'}{q}\}}.
\end{align}
\end{lemma}

The rest of this section is devoted to the proof of Lemma \ref{L:mainfio}.
To prove this lemma, we will first show that $L^{-\frac{\alpha}{2}}$ can be dominated 
by a certain dyadic operator. The dyadic operator will essentially be an 
infinite sum of dyadic versions of the classical fractional integral 
operator. In principle, we should be able to apply the results of 
\cite{LacMoePerTor2010} to each term to deduce the desired bound 
in Lemma \ref{L:mainfio}. However, as we will see, there are some 
subtleties that must be addressed. 

The operator is given by an integral operator with kernel 
$K(x,y)$. By \cite{Tang2011}*{Lemma 3.3} we know that the kernel satisfies the 
following bound: for every $\phi>0$ there $C_\phi$ such that 
\begin{align*}
\abs{K(x,y)}
\leq\frac{C_\phi}{(1+\abs{x-y}(\frac{1}{\rho(x)} +
  \frac{1}{\rho(y)}))^{\phi}}
\frac{1}{\abs{x-y}^{n-\alpha}}.
\end{align*}
Given this estimate it is now easy to dominate $L^{-\frac{\alpha}{2}}$ by a 
dyadic operator. First, fix $x\in\R^n$ and let $\phi >0$. Below, 
$Q_k^{(x)}$ is the cube of side--length $2^k$ centered at $x$. For 
non--negative $f$ there holds
\begin{align}\label{E:dyade1}
\abs{L^{-\frac{\alpha}{2}} f(x)}
&\lesssim \sum_{k\in\mathbb{Z}}\int_{Q_{k+1}^{(x)}\setminus Q_{k}^{(x)}}
  \frac{C_\phi}{(1+\abs{x-y}(\frac{1}{\rho(x)} +
  \frac{1}{\rho(y)}))^{\phi}}
  \frac{1}{\abs{x-y}^{n-\alpha}}
  f(y)dy
\\&\leq\sum_{k\in\mathbb{Z}}\int_{Q_{k+1}^{(x)}\setminus Q_{k}^{(x)}}
  \frac{C_\phi}{(1+\abs{x-y}(
  \frac{1}{\rho(y)}))^{\phi}}
  \frac{1}{\abs{x-y}^{n-\alpha}}
  f(y)dy.
\end{align}
Now, for $y\in Q_{k+1}^{(x)}\setminus Q_{k}^{(x)}$, $\abs{x-y}\simeq 
\ell (Q_k^{(x)})$ and $\rho(y)\leq \widetilde{\rho}(Q_k^{(x)})$ and so there holds 
\begin{align*}
\frac{C_\phi}{(1+\abs{x-y}(
  \frac{1}{\rho(y)}))^{\phi}}
  \frac{1}{\abs{x-y}^{n-\alpha}}
\leq \frac{C_\phi}{\widetilde{\psi}_\phi(Q_{k+1})}
  \frac{\abs{Q_{k+1}^{(x)}}^{\frac{\alpha}{n}}}{\abs{Q_{k+1}^{(x)}}}.
\end{align*}
\noindent Inserting this into \eqref{E:dyade1} we have 
\begin{align*}
\abs{L^{-\frac{\alpha}{2}} f(x)}
&\lesssim C_\phi\sum_{k\in\mathbb{Z}}\int_{Q_{k+1}^{(x)}\setminus Q_{k}^{(x)}}
  \frac{1}{\widetilde{\psi}_\phi(Q_{k+1}^{(x)})}
  \frac{\abs{Q_{k+1}^{(x)}}^{\frac{\alpha}{n}}}{\abs{Q_{k+1}^{(x)}}}f(y)dy
\\
&\leq C_\phi\sum_{k\in\mathbb{Z}}\frac{\abs{Q_{k+1}^{(x)}}^{\frac{\al}{n}}}
  {\widetilde{\psi}_\phi(Q_{k+1}^{(x)})}\avg{f}_{Q_{k+1}^{(x)}}.
\end{align*}
Now setting $\phi=\theta$ and
recalling that there is a collection of $M=M(n)$ dyadic 
lattices such that every cube $Q$ is contained in a cube $P$ from 
one of these lattices with $\ell (P) \lesssim \ell (Q)$, we deduce that 
$\abs{L^{-\frac{\alpha}{2}} f(x)}$ can be dominated by a finite sum of 
operators of the form
\begin{align}\label{D:dyadicfio}
I_{\alpha,\theta}^{\mathcal{D}}f(x)
:=\sum_{Q\in\mathcal{D}}\frac{(\ell (Q))^{\alpha}}{\widetilde{\psi}_\theta(Q)}
  \avg{f}_{Q}Q(x).
\end{align}
\noindent Lemma \ref{L:mainfio} will follow if for every dyadic 
lattice $\mathcal{D}$ we can show 
\begin{align}\label{E:mainfio2}
\norm{I_{\alpha,\theta}^{\mathcal{D}}:L^p(w^p)\to L^q(w^q)}
\lesssim [w]_{\widetilde{A}_{p,q}^{\alpha,\theta / K}}^{(1-\frac{\alpha}{n})
  \max\{1,\frac{p'}{q}\}}.
\end{align}

We now divide the cubes into collections in which we hold 
$\widetilde{\psi}_\theta(Q)$ constant. Thus, for 
$r\in\mathbb{N}$ set $\mathcal{Q}_r:=\{Q\in\mathcal{D}:\widetilde{\psi}_\theta(Q)
\simeq 2^{r\theta}\}$. Since $\widetilde{\psi}_\theta(Q) > 1$, the sum in 
\eqref{D:dyadicfio} can be written as:
\begin{align*}
I^\mathcal{D}_{\alpha,\theta} f(x)
&= \sum_{r\geq 0} \sum_{Q\in\mathcal{Q}_{r}}
  \frac{(\ell (Q))^{\alpha}}{\widetilde{\psi}_\theta(Q)}
  \avg{f}_{Q}Q(x)
\\&\simeq\sum_{r\geq 0} 2^{-r\theta}\sum_{Q\in\mathcal{Q}_{r}}
  (\ell (Q))^{\alpha}\avg{f}_{Q}Q(x)
\\&=:\sum_{r\geq 0} 2^{-r\theta}
  I_{\alpha}^{\mathcal{Q}_{r}}f(x).
\end{align*}

The operators $I_\alpha^{\mathcal{Q}_r}$ are very similar to the 
standard dyadic versions of the classical fractional integral operator. 
Indeed, the only difference is that in the classical case, 
$\mathcal{Q}_r=\mathcal{D}$. For the cubes $Q\in\mathcal{Q}_r$
\begin{align}\label{E:gesum}
\left(\frac{1}{\abs{Q}}
  \int_{Q}w(x)^qdx\right)
\left(\frac{1}{\abs{Q}}
  \int_{Q}w(x)^{-p'}dx\right)^{\frac{q}{p'}}
\leq [w,\sigma]_{A_{p,q}^{\alpha,\theta / K}}
  2^{r\left(\frac{\theta}{K}+\frac{\theta}{K}\frac{q}{p'}\right)}.
\end{align}

The point of this computation is that on the cubes in $\mathcal{Q}_r$, 
the $A_{p,q}$ characteristic is finite and so we would like to apply 
the sharp theorem of \cite{LacMoePerTor2010} to each of the operators 

The problem with this approach is that the theorem of 
Lacey--Moen--P\'erez--Torres is for the continuous version of 
the fractional integral operator. Their proof uses a sharp extrapolation 
theorem and this can not be directly applied to an operator like 
$I_\alpha^{\mathcal{Q}_r}$. On the other hand, purely dyadic versions 
of this theorem, for example \cite{Moe2012}, are only valid for 
certain values of $p$ and $q$. 

We must therefore prove a version of the theorem of Lacey--Moen--P\'erez--Torres 
for the operators $I_\alpha^{\mathcal{Q}_r}$. That is, we must 
prove the estimate
\begin{align}\label{E:toprove}
\norm{I_\alpha^{\mathcal{Q}_r}:L^p(w^p)\to L^q(w^q)}
\lesssim \left([w]_{A_{p,q}^{\alpha,\theta / K}}
  2^{r\left(\frac{\theta}{K}+
  \frac{\theta}{K}\frac{q}{p'}\right)}
  \right)^{(1-\frac{\alpha}{n})
  \max\{1,\frac{p'}{q}\}}.
\end{align}

Proving this estimate is the content of the next subsection. We will use 
a modified version of well-known extrapolation theorems. It is likely 
that the extrapolation theorem in the next subsection exists in the literature 
and we are aware of many similar theorems, but we have not been able to 
find an exact version of what we need. In any case, this will be well-known 
to experts, but we give some details; see \cite{CruMarPer2011,LacMoePerTor2010} for more 
information. 

\subsection{An Extrapolation Argument}
In this section, we will prove \eqref{E:toprove}. We will actually prove 
something slightly more general. 

Let $\mathcal{Q}$ be a finite collection of dyadic cubes. We will 
define a class of weights in the following way. We define the 
$A_{p,q}^{\mathcal{Q}}$ characteristic of a weight $w$ by
\begin{align*}
[w]_{A_{p,q}}^{\mathcal{Q}}:=
\sup_{Q\in\mathcal{Q}}
  \left(\frac{1}{\abs{Q}}\int_{Q}w^q(x)dx\right)
  \left(\frac{1}{\abs{Q}}\int_{Q}w^{-p'}(x)dx\right)^{\frac{q}{p'}}
<\infty
\end{align*}
for $1<p$ and for $p=1$
\begin{align*}
[w]_{1,q}^{\mathcal{Q}}:=
\left(\frac{1}{\abs{Q}}\int_{Q}w^q(x)dx\right)
\left(\inf_{Q} w^q(x)\right)<\infty.
\end{align*}
Define the following ``$\mathcal{Q}$--dyadic'' maximal function
\begin{align*}
M^{\mathcal{Q}}f(x)
:=\sup_{Q\in\mathcal{Q}}\frac{Q(x)}{\abs{Q}}\int_{\abs{f(x)}}dx
\end{align*}
and the ``$\mathcal{Q}$--dyadic'' fractional integral operator
\begin{align*}
I_{\alpha}^{\mathcal{Q}}f(x)
:=\sum_{Q\in\mathcal{Q}}\abs{Q}^{\alpha/n}\avg{f}_{Q}Q(x).
\end{align*}

Estimate \eqref{E:toprove} will follow from the following theorem.
\begin{theorem}\label{T:re}
Let $1<p<\frac{n}{\alpha}$ and $q$ be defined by the equation 
$\frac{1}{q}=\frac{1}{p}-\frac{\alpha}{n}$. 
For $w\in {A}_{p,q}^\mathcal{Q}$. There holds
\begin{align*}
\norm{I_{\alpha}^{\mathcal{Q}}:L^p(w^p)\to L^q(w^q)}
\lesssim 
  [w]_{{A}_{p,q}^\mathcal{Q}}^{(1-\frac{\alpha}{n})
  \max\{1,\frac{p'}{q}\}},
\end{align*}
where the implied constant depends on $p,q,\alpha$ and $n$.
\end{theorem}

We remark again that, in principle, this theorem is proven in 
\cite{LacMoePerTor2010}. However, in this setting, we are only considering 
cubes $Q\in\mathcal{Q}$ and it is not clear that their theorem can be 
quoted directly. However, their proof can be modified (in some portions, the 
proof can be quoted directly) to the present setting, and this is what we do. 

The remainder of this subsection is devoted to the proof of this 
theorem. We will use the same proof as in \cite{LacMoePerTor2010}, modified 
for our setting. The outline is as follows. We first show that it 
suffices to prove two weak--type bounds. We then prove an extrapolation 
theorem for our setting. Finally, we will prove a ``base estimate'' from 
which we can extrapolate. 

In \cite{Saw1984,Saw1988} Sawyer shows that for the fractional integral 
operator, strong--type estimates follow from weak--type estimates. He does 
this by showing in \cite{Saw1988} that the fractional integral operator 
is bounded between two weighted spaces if and only if ``testing'' holds 
(that is, if and only if the norm inequality is satisfied uniformly of 
indicators of cubes; see \cite{LacSawUri2009} for a dyadic version of 
this theorem). But in \cite{Saw1984} he shows that if $T$ is a 
self--adjoint integral operator, then testing holds if $T$ and it's adjoint 
satisfy a weak--type bound. Thus, we have the following. 
\begin{lemma}\label{L:rsaw}
Let $w$ be a weight, $0<\alpha<n$, and 
$1<p\leq q<\infty$.
Then the operator norm 
\begin{align}\label{E:wt}
\norm{I_\alpha^\mathcal{Q}: L^p(w^p)\to L^q(w^q)}
\end{align}
is controlled by
\begin{align*}
\norm{I_\alpha^\mathcal{Q}: L^p(w^p)\to L^{q,\infty}(w^q)}
  + \norm{I_\alpha^\mathcal{Q}: L^{q'}(w^{-q'}) \to L^{p',\infty}(w^{-p'})}.
\end{align*}
\end{lemma}

Given Lemma \ref{L:rsaw}, we now turn our attention to proving 
the following lemma. 
\begin{lemma}\label{L:sharpweak}
Let $1<p<\frac{n}{\alpha}$ and $q$ be defined by the equation 
$\frac{1}{q}=\frac{1}{p}-\frac{\alpha}{n}$. 
For $w\in {A}_{p,q}$. There holds
\begin{align*}
\norm{I_\alpha^\mathcal{Q}: L^p(w^p)\to L^{q,\infty}(w^q)}
+ 
\norm{I_\alpha^\mathcal{Q}: L^{q'}(w^{-q'}) \to L^{p',\infty}(w^{-p'})}
\lesssim 
  [w]_{{A}_{p,q}}^{(1-\frac{\alpha}{n})},
\end{align*}
where the implied constant depends on $p,q,\alpha$ and $n$.
\end{lemma}
%

We first state the extrapolation theorem. It is our version of the extrapolation 
theorem (Theorem 2.1 in \cite{LacMoePerTor2010}). 
\begin{theorem}\label{T:extrap}
Suppose that $T$ is an operator defined on $C_c^\infty$. Suppose 
that $1\leq p_0\leq q_0<\infty$ and that 
\begin{align*}
\norm{Tf}_{L^{q_0}(w^{q_0})}
\lesssim [w]_{A_{p_0,q_0}^{\mathcal{Q}}}^{\gamma}
  \norm{f}_{L^{p_0}(w^{p_0})}
\end{align*}
for all $w\in A_{p_0,q_0}^{\mathcal{Q}}$ and some $\gamma > 0$. Then 
\begin{align*}
\norm{Tf}_{L^q(w^q)}
\lesssim [w]_{A_{p,q}^{\mathcal{Q}}}^{\gamma\max\{1,
  \frac{q_0}{p_0'}\frac{p'}{q'}\}}
  \norm{f}_{L^p(w^p)}
\end{align*}
holds for all $p,q$ satisfying $\frac{1}{p}-\frac{1}{q}
=\frac{1}{p_0}-\frac{1}{q_0}$ and all $w\in A_{p,q}^{\mathcal{Q}}$. 
\end{theorem}

As is familiar to experts, the key to proving Theorem \ref{T:extrap} 
is a version of the Rubio de Francia iteration algorithm. Once we have 
established this iteration algorithm, we can prove the 
extrapolation theorem. We follow the proof in \cite{Gra2004}. 
Below, $A_{p}^{\mathcal{Q}}$ is the $A_p$ class adapted to $\mathcal{Q}$:
\begin{align*}
[w]_{A_p^{\mathcal{Q}}}
:=\sup_{Q\in\mathcal{Q}}\left(\frac{1}{\abs{Q}}\int_{Q}w(x)dx\right)
  \left(\frac{1}{\abs{Q}}\int_{Q}
  w^{-\frac{p'}{p}}(x)dx\right)^{\frac{p}{p'}}.
\end{align*}

\begin{lemma}\label{L:rsa}
Suppose that $1\leq r_0 < r$, $v\in A_r^{\mathcal{Q}}$, and $g$ is a 
non--negative function in $L^{(r/r_0)'}(v)$. Then there is a 
function $G$ such that
\begin{itemize}
 \item[(a)] $g\leq G$;
 \item[(b)] $\norm{G}_{L^{(r/r_0)'}(v)}
  \lesssim \norm{g}_{L^{(r/r_0)'}(v)}$;
 \item[(c)]$Gv\in A_{r_0}^{\mathcal{Q}}$ with 
  $[Gv]_{A_{r_0}^{\mathcal{Q}}}\lesssim [v]_{A_r^{\mathcal{Q}}}$.
\end{itemize}
The implied constants are independent of $r_0,r,\mathcal{Q},v,G$ and $g$. 
\end{lemma}
\begin{proof}
Let $t=\frac{r'}{(r/r_0)'}=\frac{r-r_0}{r-1}$. Note that since 
$1\leq r_0 < r$ there holds $0<t\leq 1$. Define
\begin{align*}
Rg
:=\left(M^\mathcal{Q}(g^{\frac{1}{t}}v)v^{-1}\right)^{t}.
\end{align*}
We compute the norm of $R$ as an operator from $L^{\frac{r}{r-r_0}}(v)$
to itself. Let $f\in L^{\frac{r}{r-r_0}}(v)$. There holds
\begin{align*}
\int_{\R^n}(Rg)(x)^{\frac{r}{r-r_0}}v(x)dx
&=\int_{\R^n}\left(M^\mathcal{Q}
  (g^{\frac{1}{t}}v)v^{-1}\right)^{t\frac{r}{r-r_0}}(x)v(x)dx
\\&=\int_{\R^n}[M^\mathcal{Q}(g^{\frac{1}{t}}v)(x)]^{t\frac{r}{r-r_0}}
  v(x)^{1-t\frac{r}{r-r_0}}dx
\\&=\int_{\R^n}[M^\mathcal{Q}(g^{\frac{1}{t}}v)(x)]^{r'}
  v(x)^{-\frac{r'}{r}}dx.
\end{align*}
Now $M^\mathcal{Q}$ is bounded from $L^{r'}(v^{-\frac{r'}{r}})$ 
to itself with 
norm $[v^{-\frac{r'}{r}}]_{A_{r'}^{\mathcal{Q}}}^{1/(r'-1)}$.
Thus we can continue the estimate with
\begin{align*}
\int_{\R^n}[M^\mathcal{Q}(g^{\frac{1}{t}}v)(x)]^{r'}
  v(x)^{-\frac{r'}{r}}dx
&\leq [v^{-\frac{r'}{r}}]_{A_{r'}^{\mathcal{Q}}}^{\frac{r'}{r'-1}}
  \int_{\R^n}g(x)^{\frac{r'}{t}}v(x)^{r'}v(x)^{-\frac{r'}{r}}dx
\\&=[v^{-\frac{r'}{r}}]_{A_{r'}^{\mathcal{Q}}}^{\frac{r'}{r'-1}}
  \int_{\R^n}g(x)^{\frac{r}{r-r_0}}v(x)dx.
\end{align*}
And so we have 
\begin{align*}
\norm{R:L^\frac{r}{r-r_0}(v)\to L^\frac{r}{r-r_0}(v)}
\leq [v^{-\frac{r'}{r}}]_{A_{r'}^{\mathcal{Q}}}^{t(r-1)}
=[v]_{A_r^{\mathcal{Q}}}^{t}.
\end{align*}
Define 
\begin{align*}
G
:=\sum_{k=0}^{\infty}\frac{R^k(g)}{2^k\norm{R}^k},
\end{align*}
where $\norm{R}:=\norm{R:L^\frac{r}{r-r_0}(v)\to L^\frac{r}{r-r_0}(v)}$, 
and $R^0=\textnormal{Id}$. Then $g\leq G$ and there holds
\begin{align*}
\norm{G}_{L^{\frac{r}{r-r_0}}(v)}
\leq \sum_{k=0}^{\infty}\frac{\norm{R}^k\norm{g}_{L^{\frac{r}{r-r_0}}(v)}}
  {2^k\norm{R}^k}
\simeq \norm{g}_{L^{\frac{r}{r-r_0}}(v)}.
\end{align*}
Noting that $(r/r_0)'=r/(r-r_0)$ we see that (a) and (b) are proven.

We now need to estimate $[Gv]_{A_{r_0}^{\mathcal{Q}}}$. First, 
by applying $R$ to $G$ we have
\begin{align*}
RG
=\sum_{k=0}^{\infty}\frac{R^{k+1}(g)}{2^k\norm{R}^k}
=2\norm{R}\sum_{k=0}^{\infty}\frac{R^{k+1}(g)}{2^{k+1}\norm{R}^{k+1}}
\leq 2\norm{R}G.
\end{align*}
Thus 
\begin{align*}
\left(M^\mathcal{Q}(G^{\frac{1}{t}}v)v^{-1}\right)^{t}
\lesssim \norm{R}G
\lesssim [v]_{A_r^{\mathcal{Q}}}^{t}G.
\end{align*}
Taking $t^{\textnormal{th}}$ roots and rearranging we see that
\begin{align*}
\left(M^\mathcal{Q}(G^{\frac{1}{t}}v)\right)
\lesssim \left(G^{\frac{1}{t}}v\right)[v]_{A_r^{\mathcal{Q}}}.
\end{align*}
Thus for all cubes $Q\in\mathcal{Q}$ we have
\begin{align*}
\frac{1}{\abs{Q}}\int_{Q}G^\frac{1}{t}(x)v(x) dx
\lesssim [v]_{A_r^{\mathcal{Q}}} G^{\frac{1}{t}}v.
\end{align*}
Again rearranging this implies 
\begin{align}\label{E:ite1}
G \gtrsim
[v]_{A_r^{\mathcal{Q}}}^{-t}
w^{-t}\left(\frac{1}{\abs{Q}}\int_{Q}G(x)^{\frac{1}{t}}v(x)dx\right)^{t}.
\end{align}
We now estimate the $A_{r_0}^{\mathcal{Q}}$ characteristic of $Gv$. 
We need to estimate 
\begin{align}\label{E:ite2}
\left(\frac{1}{\abs{Q}}\int_{Q}G(x)v(x)dx\right)
\left(\frac{1}{\abs{Q}}\int_{Q}G(x)^{-\frac{1}{r_0-1}}
  v(x)^{-\frac{1}{r_0-1}}dx\right)^{r_0-1}.
\end{align}
By H\"older's Inequality with exponents 
$(1/t)$ and $(1/t)'=1/t-1$ it follows that the first factor is 
dominated by
\begin{align}\label{E:ite3}
\left(\frac{1}{\abs{Q}}\int_{Q}G(x)^{\frac{1}{t}}v(x)dx\right)^{t}
\left(\frac{1}{\abs{Q}}\int_{Q}v(x)dx\right)^{t-1}.
\end{align}
By \eqref{E:ite1} the second factor is controlled by
\begin{align}\label{E:ite4}
\left(\frac{1}{\abs{Q}}\int_{Q}
  \left\{
  [v]_{A_r^{\mathcal{Q}}}^{-1}
  \left(\frac{1}{\abs{Q}}\int_{Q}
  G(y)^{\frac{1}{t}}v(y)dy\right)
  \right\}^{-\frac{t}{r_0-1}}
  v^{\frac{t}{r_0-1}}(x)v(x)^{-\frac{1}{r_0-1}}dx\right)^{r_0-1}.
\end{align}
Multiplying \eqref{E:ite3} and \eqref{E:ite4} together, and using 
the fact that $\frac{t-1}{r_0-1}=-\frac{1}{r-1}$ we see 
that \eqref{E:ite2} is controlled by
\begin{align*}
[v]_{A_r^\mathcal{Q}}^{t}
\left\{
\left(\frac{1}{\abs{Q}}\int_{Q}v(x)dx\right)
\left(\frac{1}{\abs{Q}}\int_{Q}v(x)^{-\frac{1}{r-1}}dx
  \right)^{r-1}
\right\}^{1-t}
=[v]_{A_r^\mathcal{Q}}^{t}[v]_{A_r^\mathcal{Q}}^{1-t}
=[v]_{A_r^\mathcal{Q}}.
\end{align*}
This proves (c).
\end{proof}

\begin{remark}
We now discuss the proof Theorem \ref{T:extrap}. Given the iteration 
algorithm Lemma \ref{L:rsa}, the proof of Theorem \ref{T:extrap} is 
\textit{exactly} the same as the proof of \cite{LacMoePerTor2010}*{Theorem 2.1}.
We will not restate the proof, but we will explain why it is true. 

It is a general principle that given an iteration algorithm like in 
Lemma \ref{L:rsa}, the extrapolation theorem will follow. The main 
idea in a proof of the extrapolation theorem is to factor expressions like  
$\abs{g(x)}w(x)^q$ into pieces on which the ``base case'' bound 
can be used. 

The extrapolation argument is not very sensitive to the operator. 
For example, we do not need to assume that the operator is linear 
or even sub--linear; we only need to assume that it is defined on 
(for example) $C_c^\infty$, smooth functions with compact support. 
The fact that we only know data about $w$ for the cubes $Q$ might 
seem insufficient to deduce the claimed bounds, but we are assuming 
that the operator is bounded for the base case exponents, and this 
gives enough information to deduce the claimed bounds. \qed
\end{remark}

Using Theorem \ref{T:extrap} we have the following corollary.
The proof is in \cite{LacMoePerTor2010}*{Corollary 2.2}.
\begin{corollary}\label{c:wke}
Suppose that for some $1\leq p_0\leq q_0<\infty$, an operator 
$T$ satisfies the weak--type $(p_0,q_0)$ inequality:
\begin{align*}
\norm{T:L^{p_0}(w^{p_0})\to L^{q_0,\infty}(w^{q_0})}
\leq c[w]_{A_{p_0,q_0}^\mathcal{Q}}^\gamma
\end{align*}
for every $w\in A_{p_0,q_0}$ and some $\gamma>0$. Then $T$ also 
satisfies the weak--type $(p,q)$ inequality 
\begin{align*}
\norm{T:L^p(w^p)\to L^{q,\infty}(w^q)}
\leq c[w]_{A_{p,q}^\mathcal{Q}}^{\gamma\max\{1,\frac{q_0}{p_0'}
  \frac{p'}{q}\}}
\end{align*}
for all $1<p\leq q < \infty$ that satisfy
\begin{align*}
\frac{1}{p}-\frac{1}{q}
=\frac{1}{p_0}-\frac{1}{q_0}
\end{align*}
and all $w\in A_{p,q}$. 
\end{corollary}

We now prove a ``base case'' weak-type estimate from which we can use Corollary 
\ref{c:wke} to extrapolate to all exponents. 
\begin{lemma}\label{L:bcextrap}
Let $q_0=n/(n-\alpha)=(n/\alpha)'$.
There holds
\begin{align}\label{E:bcextrap}
\norm{I_{\alpha}^{\mathcal{Q}}f}_{L^{q_0,\infty}(w^{q_0})}
\lesssim [w]_{A_{1,q_0}^{\mathcal{Q}}}^{1-\alpha / n}
  \norm{f}_{L^{1}(w)}
\end{align}
for any weight $w$. 
\end{lemma}
\begin{proof}
For convenience let $u=w^{q_0}$. 
Let $\mathcal{Q}_{M}$ denote the maximal cubes in $\mathcal{Q}$.
Recall that we assume that $\mathcal{Q}$ is finite so every cube 
in $\mathcal{Q}$ is contained in a unique cube in $\mathcal{Q}_M$. 
For every $Q\in\mathcal{Q}_M$ we will prove
\begin{align}\label{E:bcextrap1}
\norm{QI_{\alpha}^{\mathcal{Q}}f}_{L^{q_0,\infty}(u)}
\lesssim \norm{Qf}_{L^1(M^\mathcal{Q}u)^{1/q_0}}.
\end{align}
This will imply \eqref{E:bcextrap} by the following argument. Now, 
for every $Q\in\mathcal{Q}_M$ the $A_{1,q_0}^\mathcal{Q}$ condition implies 
$M^{\mathcal{Q}}u(x) \leq [w]_{A_{1,q_0}^{\mathcal{Q}}}u(x)$. We 
therefore have
\begin{align*}
\norm{I_{\alpha}^{\mathcal{Q}}f}_{L^{q_0,\infty}(u)}
&\leq\sum_{Q\in\mathcal{Q}_M}
  \norm{QI_{\alpha}^{\mathcal{Q}}f}_{L^{q_0,\infty}(u)}
\\&\leq\sum_{Q\in\mathcal{Q}_M}\int_{Q}\abs{f(x)}
  (M^\mathcal{Q} u)^{\frac{1}{q_0}}(x)dx
\\&\leq[w]_{A_{1,q_0}^{\mathcal{Q}}}^{\frac{1}{q_0}}
  \int_{\R^d}\abs{f(x)}
  u(x)^{\frac{1}{q_0}}(x)dx
\\&=[w]_{A_{1,q_0}^{\mathcal{Q}}}^{1-\alpha / n}
  \int_{\R^d}\abs{f(x)}w(x)dx.
\end{align*}
Now, fix a cube in $P\in\mathcal{Q}_M$. Observe that there 
holds $$P(x)I_\alpha^{\mathcal{Q}}f(x)=
\sum_{Q\in\mathcal{Q}:Q\subset P}\abs{Q}^{\alpha / n}
\avg{Pf}_{Q}Q(x). $$

Note that $I_{\alpha}^{\mathcal{Q}}$ can be written as an integral 
operator with kernel $K(x,y):=\sum_{Q\in\mathcal{Q}:Q\subset P}
\frac{Q(x)Q(y)}{\abs{Q}^{1-\alpha / n}}$. Thus using Minkowski's inequality 
for the $L^{q,\infty}$ norm, there holds
\begin{align*}
\norm{\sum_{Q\in\mathcal{Q}:Q\subset P}
  \abs{Q}^{\alpha / n}\avg{Pf}_{Q}Q(x)}_{L^{q_0,\infty}(u)}
\end{align*}
is dominated by 
\begin{align}\label{E:bcextrap2}
\int_{P}\abs{f(y)}
  \norm{\sum_{Q\in\mathcal{Q}:Q\subset P}\frac{Q(x)Q(y)}
  {\abs{Q}^{1-\alpha / n}}}_{L^{q_0,\infty}(u)}dy.
\end{align}
Now, we compute the $L^{q_0,\infty}(u)$ norm inside the integral. 
Let $\lambda > 0$ and let $\mathcal{Q}_{\lambda}$ be the 
maximal cubes in $\mathcal{Q}$ with $\abs{Q}^{1-\alpha / n} < 
\lambda^{-1}$. Now, for a fixed $x$, 
$\sum_{Q\in\mathcal{Q}:Q\subset P}Q(x)\abs{Q}^{\alpha / n-1}$ is 
a geometric sum. Thus, if 
$\sum_{Q\in\mathcal{Q}:Q\subset P}Q(x)\abs{Q}^{\alpha / n-1} > \lambda$, 
then $x$ is contained in a unique element of $\mathcal{Q}_\lambda$. 
Now, let $Q_\lambda(y)$ denote the unique element of $\mathcal{Q}_\lambda$ 
that contains $y$ (if there is such an element). Note also that 
$\lambda < \abs{Q}^{\alpha / n -1}=\abs{Q}^{-\frac{1}{q}}$. Using this 
notation and these observations there holds
\begin{align*}
\lambda\left(u\left\{x:\sum_{Q\in\mathcal{Q}:Q\subset P}\frac{Q(x)Q(y)}
  {\abs{Q}^{1-\alpha / n}} > \lambda\right\}\right)^{\frac{1}{q}}
&= \lambda u(Q_{\lambda}(y))^{\frac{1}{q}}
\\&\leq \frac{1}{\abs{Q_\lambda(y)}^{1-\alpha / n}}
  u(Q_{\lambda}(y))^{\frac{1}{q}}
\\&=\left(\frac{1}{\abs{Q_\lambda(y)}}
  u(Q_{\lambda}(y))\right)^{\frac{1}{q}}.
\end{align*}
Taking a supremum over $\lambda > 0$ we deduce that 
\begin{align*}
\norm{\sum_{Q\in\mathcal{Q}:Q\subset P}\frac{Q(x)Q(y)}
  {\abs{Q}^{1-\alpha / n}}}_{L^{q_0,\infty}(u)}
\leq (M^\mathcal{Q}u(y))^{\frac{1}{q}}. 
\end{align*}
Inserting this into \eqref{E:bcextrap2} will give 
\eqref{E:bcextrap1}.
\end{proof}

We are now in a position to prove Lemma \ref{L:sharpweak}.
Using extrapolation, we know that
\begin{align*}
\norm{I_\alpha^\mathcal{Q}: L^p(w^p)\to L^{q,\infty}(w^q)}
\lesssim [w]_{A_{p,q}^{\mathcal{Q}}}^{1-\frac{\al}{n}}
\end{align*}
and
\begin{align*}
\norm{I_\alpha^\mathcal{Q}: L^{q'}(w^{-q'}) \to L^{p',\infty}(w^{-p'})}
\lesssim [w^{-1}]_{A_{q',p'}^\mathcal{Q}}^{1-\frac{\al}{n}}.
\end{align*}
Now, $[w^{-1}]_{A_{p',q'}^\mathcal{Q}}=[w]_{A_{p,q}^\mathcal{Q}}^{\frac{p'}{q}}$ and 
$[w]_{A_{p,q}^\mathcal{Q}}>1$ so there holds
\begin{align*}
\norm{I_\alpha^\mathcal{Q}: L^p(w^p)\to L^{q,\infty}(w^q)}
\lesssim [w]_{A_{p,q}^{\mathcal{Q}}}^{1-\frac{\al}{n}}
  +[w^{-1}]_{A_{q',p'}^\mathcal{Q}}^{1-\frac{\al}{n}}
\lesssim [w]_{A_{p,q}^{\mathcal{Q}}}^{(1-\frac{\al}{n})
  \max\{1,\frac{p'}{q}\}}.
\end{align*}
Thus the proof of Lemma of \ref{L:sharpweak} is complete and 
so we have proved Theorem \ref{T:mainfio}.

\section{Weights Associated to $L$ and  Connections to BMO Space Associated to $L$}
\label{s:wtdthry}

\setcounter{equation}{0}

In this section, we recall the definition and properties of the BMO space ${\rm BMO}_\infty(\mathbb R^n)$ associated to 
$L$. Then we build the exp-log connection of  $A_p^\infty$ and ${\rm BMO}_\infty(\mathbb R^n)$.


For any $\theta \geq 0$ we can define the following $\BMO_{\theta}(\mathbb R^n)$ space as the 
set of functions such that
\begin{align}\label{D:bmotheta}
\norm{f}_{\BMO_{\theta}(\mathbb R^n)}
:= \sup_{Q\textnormal{ a cube}}
  \frac{1}{\psi_{\theta}(Q)\abs{Q}}
  \int_{Q}\abs{f(y) - \avg{f}_{Q}} dy
< \infty. 
\end{align}
We also have the following  $\BMO_{\infty}(\mathbb R^n)$ space 
\begin{align}\label{D:bmoinfty}
\BMO_{\infty}(\mathbb R^n):=\cup_{\theta\geq0} \BMO_{\theta}(\mathbb R^n). 
\end{align}

Based on the definition of $\BMO_{\infty}(\mathbb R^n)$, we  provide the proof of Theorem \ref{T:bmoas}.

\smallskip
%
%

\begin{proof}[\bf Proof of Theorem \ref{T:bmoas}]

Proof of (i):

Suppose that $w\in A_{p}^{\infty}$. Then there exists a $\theta\geq0$ such that
$w\in A_{p}^{\theta}$. Let $\varphi=\log w$ and 
$\mu = \log\left(\left(\frac{1}{w}\right)^{\frac{1}{p-1}}\right)
=\frac{-\varphi}{p-1}$. Then for any cube $Q$ we have 
$e^{\avg{\varphi}_{Q}}e^{(p-1)\avg{\mu}_{Q}}=1$ and so we can write the 
$A_{p}^{\theta}$ condition for $w$ as follows:
\begin{align}\label{E:apexp}
\frac{1}{\psi_{\theta}(Q)^{p}}
\left(\frac{1}{\abs{Q}}\int_{Q}e^{\varphi(x) - \avg{\varphi}_{Q}}dx\right)
\left(\frac{1}{\abs{Q}}\int_{Q}e^{\mu(x)-\avg{\mu}_{Q}}dx\right)^{p-1}\leq [w]_{A_p^\theta} < \infty.
\end{align}
By Jensen's inequality we have
\begin{align*}
\frac{1}{\abs{Q}}\int_{Q}e^{\varphi(x) - \avg{\varphi}_{Q}}dx  \geq 1
\hspace{.2in}
\textnormal{and}
\hspace{.2in}
\frac{1}{\abs{Q}}\int_{Q}e^{\mu(x) - \avg{\mu}_{Q}}dx \geq 1.
\end{align*}
Thus, noting that $\psi_{\theta}(Q)^p = \psi_{p\theta}(Q)$, we 
conclude that for any $w\in A_{p}^{\theta}$ we have
\begin{align}\label{E:bothbdd}
&\frac{1}{\psi_{p\theta}(Q)}
\left(\frac{1}{\abs{Q}}\int_{Q}e^{\varphi(x) - \avg{\varphi}_{Q}}dx\right)\leq {[w]_{A_p^\theta}  \over \left(\frac{1}{\abs{Q}}\int_{Q}e^{\mu(x)-\avg{\mu}_{Q}}dx\right)^{p-1}}
\leq [w]_{A_p^\theta},
\end{align}
and similarly,
\begin{align}\label{E:bothbdd2}
&\frac{1}{\psi_{p\theta}(Q)}
\left(\frac{1}{\abs{Q}}\int_{Q}
  e^{-(\varphi(x)-\avg{\varphi}_{Q})/(p-1)}dx\right)^{p-1}
\leq [w]_{A_p^\theta}.
\end{align}
Now for a cube $Q$, let $Q_+:=\{x\in Q: \varphi - \avg{\varphi}_{Q} \geq 0\}$ and 
$Q_{-}=Q\setminus Q_{+}$. Then we have
\begin{align}\label{E: BMO norm}
&\frac{1}{\psi_{p\theta}(Q)\abs{Q}}
\int_{Q}\abs{\varphi(x) - \avg{\varphi}_{Q}}dx\\
&=\frac{1}{\psi_{p\theta}(Q)\abs{Q}}
\left(\int_{Q_+} \big(\varphi(x) - \avg{\varphi}_{Q}\big)dx+\int_{Q_-} -\big(\varphi(x) - \avg{\varphi}_{Q}\big)dx\right)\nonumber.
\end{align}
For the first term in the right-hand side of the equality above, using the trivial estimate $t\leq e^t$, we obtain that 
\begin{align}\label{E:term1}
\frac{1}{\psi_{p\theta}(Q)\abs{Q}}
\int_{Q_+} \big(\varphi(x) - \avg{\varphi}_{Q}\big)dx&\leq 
\frac{1}{\psi_{p\theta}(Q)\abs{Q}}
  \int_{Q_{+}}e^{\varphi(x) - \avg{\varphi}_{Q}}dx\\
  &\leq 
\frac{1}{\psi_{p\theta}(Q)\abs{Q}}
  \int_{Q}e^{\varphi(x) - \avg{\varphi}_{Q}}dx \nonumber\\
  &\leq  [w]_{A_p^\theta},\nonumber
\end{align}
where the last inequality follows from \eqref{E:bothbdd}.

Now for the second term, we first consider the case $p-1\leq1$. Then using the trivial estimate $t\leq e^t$ again we get
\begin{align}
&\frac{1}{\psi_{p\theta}(Q)\abs{Q}}\int_{Q_-} -\big(\varphi(x) - \avg{\varphi}_{Q}\big)dx\\
&\leq 
\frac{1}{\psi_{p\theta}(Q)\abs{Q}}
  \int_{Q_{-}} e^{-(\varphi(x) - \avg{\varphi}_{Q})} dx 
\\
&=
\frac{1}{\psi_{p\theta}(Q)\abs{Q}}
  \int_{Q_{-}}\Big[e^{-(\varphi(x) - \avg{\varphi}_{Q})/(p-1)}\Big]^{p-1}dx 
\\
&\leq 
\frac{1}{\psi_{p\theta}(Q)\abs{Q}}
  \int_{Q}\Big[e^{-(\varphi(x) - \avg{\varphi}_{Q})/(p-1)}\Big]^{p-1}dx 
\\
&\leq 
\frac{1}{\psi_{p\theta}(Q)}
  \left({1\over \abs{Q}}\int_{Q}e^{-(\varphi(x) - \avg{\varphi}_{Q})/(p-1)}dx \right)^{p-1}\\
  &\leq  [w]_{A_p^\theta},
\end{align}
where 
the third inequality follows from H\"older's inequality and the last inequality follows from \eqref{E:bothbdd2}.

We now consider the case $p-1>1$.  Again we have
\begin{align}\label{eeee1}
\frac{1}{\psi_{p\theta}(Q)\abs{Q}}
\int_{Q_-} -\big(\varphi(x) - \avg{\varphi}_{Q}\big)dx
&= 
\frac{p-1}{\psi_{p\theta}(Q)\abs{Q}}
\int_{Q_-} -\frac{\big(\varphi(x) - \avg{\varphi}_{Q}\big)}{p-1}dx
\\&\leq\frac{p-1}{\psi_{p\theta}(Q)\abs{Q}}
  \int_{Q_{-}}e^{-(\varphi(x) - \avg{\varphi}_{Q})/p-1}dx.
\end{align}

Next, we note that $\psi_\theta(Q)\geq1$ for all $Q$ and $\theta>0$, and that $p-1>1$. Thus we have
$$  \psi_{p\theta}(Q)^{1\over p-1} \leq  \psi_{p\theta}(Q), $$ 
which implies that
$$ {1\over \psi_{p\theta}(Q)}\leq {1\over \psi_{p\theta}(Q)^{1\over p-1} }. $$ 
Combing the above estimate and the inequality \eqref{eeee1}, we get
\begin{align*}
\frac{1}{\psi_{p\theta}(Q)\abs{Q}}
 \int_{Q_{-}}-\big(\varphi(x) - \avg{\varphi}_{Q}\big)dx
&\leq\frac{p-1}{\psi_{p\theta}(Q)^{1\over p-1}\abs{Q}}
  \int_{Q}e^{-(\varphi(x) - \avg{\varphi}_{Q})/p-1}dx\\
&\leq (p-1) [w]_{A_p^\theta}^{1\over p-1},
\end{align*}
where the last inequality follows from \eqref{E:bothbdd2}.

Now combining the estimates of the first and second terms on the right-hand side of \eqref{E: BMO norm}, 
we obtain that 
\begin{align*}
&\frac{1}{\psi_{p\theta}(Q)\abs{Q}}
\int_{Q}\abs{\varphi(x) - \avg{\varphi}_{Q}}dx
\leq [w]_{A_p^\theta} \max\big\{  [w]_{A_p^\theta}  ,  (p-1)[w]_{A_p^\theta}^{1\over p-1} \big\}.
\end{align*}
Hence we obtain that
 $\log w\in \BMO_{p\theta} \subset \BMO_{\infty}$, which implies that
 (i) holds.

Proof of (ii). Consider $L=-\Delta+1$ on $\mathbb R^n$. Then from \cite{BGS1} it is known that $b(x) = |x_j|$, $1\leq j\leq n$ is in $\BMO_{\infty}$. However, $e^{\delta |x_j|}$ is not in $A_{p}^{\infty}$ for any $\delta>0$ and $p\in [1,\infty)$.
 %
%
%
%
%
\end{proof}

\section{Conclusion}\label{s:con}
We briefly mention some two weight inequalities for the fractional 
integral operator $L^{-\frac{\alpha}{2}}$. Recall that $L^{-\frac{\alpha}{2}}$ is dominated by a 
finite sum of operators of the form
\begin{align*}
I_{\alpha,\theta}^{\mathcal{D}}
:=\sum_{Q\in\mathcal{D}}\frac{(\ell (Q))^{\alpha}}{\widetilde{\psi}_\theta(Q)}
  \avg{f}_{Q}Q(x).
\end{align*}
And by setting $\mathcal{Q}_r:=\{Q\in\mathcal{D}:\bp{\theta}(Q)\simeq 2^{r\theta}\}$
we can further decompose $I_{\alpha,\theta}^{\mathcal{D}}$ as
\begin{align*}
I^\mathcal{D}_{\alpha,\theta} f(x)
&= \sum_{r\geq 0} \sum_{Q\in\mathcal{Q}_{r}}
  \frac{(\ell (Q))^{\alpha}}{\widetilde{\psi}_\theta(Q)}
  \avg{f}_{Q}Q(x)
\\&\simeq\sum_{r\geq 0} 2^{-r\theta}\sum_{Q\in\mathcal{Q}_{r}}
  (\ell (Q))^{\alpha}\avg{f}_{Q}Q(x)
\\&=:\sum_{r\geq 0} 2^{-r\theta}
  I_{\alpha}^{\mathcal{Q}_{r}}f(x).
\end{align*}
Therefore, to establish a two weight bound, it will be enough to give a 
two weight bound for the operators $I_{\alpha}^{\mathcal{Q}_{r}}$. We also 
note that if $v$ is a weight and $\sigma:=v^{-\frac{p'}{p}}$ then there 
holds
\begin{align*}
\norm{T:L^p(v)\to L^q(w)}
=\norm{T(\sigma\cdot):L^p(\sigma)\to L^q(w)}.
\end{align*}

The following was proven by one of us and Scott Spencer \cite{RahSpe2015}.
Below, for a weight $w$ we define 
\begin{align*}
\rho_w(Q):=\frac{1}{w(Q)}\int_{Q}(M(wQ))(x)dx.
\end{align*}

\begin{lemma}\label{L:rs}
Let $1 < p \leq q < \infty$ and $\sigma,w$ be two weights. Let 
$\epsilon_{p}$ be a monotonic function on $(1,\infty)$ that satisfies 
$\int_{1}^{\infty}\frac{dt}{t\epsilon_{p}^{p}}(t)=1$ and similarly for 
$\epsilon_{q'}$. Define 
\begin{align*}
\beta(Q):=
\frac{\sigma(Q)^{\frac{1}{p'}}w(Q)^{\frac{1}{q}}}
  {\abs{Q}^{1-\frac{\alpha}{n}}}
  \rho_\sigma(Q)^{\frac{1}{p}}
  \epsilon_{p}(\rho_\sigma(Q))
  \rho_w(Q)^{\frac{1}{q'}}
  \epsilon_{q'}(\rho_w(Q))
\end{align*}
and set $[\sigma,w]_{p,q,\alpha,r}:=\sup_{Q\in\mathcal{Q}_r}\beta(Q)$. Then 
$\norm{I_{\alpha}^{\mathcal{Q}_{r}}(\sigma\cdot):L^p(\sigma)\to L^q(w)}
\lesssim [\sigma,w]_{p,q,\alpha,r}$.
\end{lemma}
Now, define 
\begin{align*}
[\sigma,w]_{p,q,\alpha}^{(\theta)}
:=\sup_{Q\textnormal{ a cube}}
  \frac{\sigma(Q)^{\frac{1}{p'}}w(Q)^{\frac{1}{q}}}
  {\psi_{\theta/2}(Q)\abs{Q}^{1-\frac{\alpha}{n}}}
  \rho_\sigma(Q)^{\frac{1}{p}}
  \epsilon_{p}(\rho_\sigma(Q))
  \rho_w(Q)^{\frac{1}{q'}}
  \epsilon_{q'}(\rho_w(Q)).
\end{align*}
The conclusion in Lemma \ref{L:rs} can be stated as
\begin{align*}
\norm{I_{\alpha}^{\mathcal{Q}_{r}}(\sigma\cdot):L^p(\sigma)\to L^q(w)}
\lesssim 2^{r\theta / 2}[\sigma,w]_{p,q,\alpha}^{(\theta)}. 
\end{align*}
Thus using Lemma \ref{L:rs} 
and the decomposition of $I_\alpha^\mathcal{D}$ we have the following 
theorem 
\begin{theorem}\label{T:twt}
With definitions as above, there holds
\begin{align*}
\norm{L^{-\frac{\alpha}{2}}(\sigma\cdot):L^p(\sigma)\to L^q(w)}
\lesssim [\sigma,w]_{p,q,\alpha}^{(\theta)}.
\end{align*}
\end{theorem}
See other results in \cite{RahSpe2015,Cru2015,CruMoe2013} to deduce similar 
two weight results in the present setting.  

The condition $[\sigma,w]_{p,q,\alpha}^{(\theta)}$ may seem to be 
complicated beyond the point of usability. Conditions like this are 
known as ``bump'' conditions. These bump conditions were introduced in 
\cite{TreVol2015} and studied more in \cites{LacSpe2015,RahSpe2015} and 
are typically smaller than other bump conditions such as Orlicz norms (this 
was shown by Treil and Volberg in \cite{TreVol2015}). For more information 
about two weight inequalities for the fractional integral operator, 
see \cites{Cru2015,CruMoe2013}.

Theorem \ref{T:twt} has a deficiency. The quantity $\rho_w(Q)$ is related
to the $A_\infty$ characteristic of a weight. In particular, 
$[w]_{A_\infty}:=\sup_{Q}\rho_w(Q)$. This is an important characteristic in 
the classical weighted theory. However, it is too large to capture enough 
information for weights in our classes. It will be interesting to develop 
an $A_\infty$ theory adapted to the operator $-\Delta + V$.

\bigskip
\noindent{\bf Acknowledgement: } The authors would like to thank Julian Bailey in Australian National University for 
pointing out errors in Section \ref{s:wtdthry} and to the statement of Theorem \ref{T:bmoas}.


\begin{bibsection}
\begin{biblist}

\bib{A}{article}{
author={Auscher, Pascal},
title={On necessary and sufficient conditions for $L^p$-estimates of Riesz transforms associated to elliptic operators on $\mathbb R^n$ and related estimates},
journal={Mem. Amer. Math. Soc.},
volume={186},
date={2007}
number={871},
pages={xviii+75 pp},

}


\bib{Bez2008}{article}{
   author={Beznosova, Oleksandra V.},
   title={Linear bound for the dyadic paraproduct on weighted Lebesgue space
   $L_2(w)$},
   journal={J. Funct. Anal.},
   volume={255},
   date={2008},
   number={4},
   pages={994--1007}
}

\bib{BGS}{article}{
    author={Bongioanni},
    author={Harboure},
    author={Salinas},
    title={Classes of weights related to Schr\"odinger operators},
    journal={J. Math. Anal. Appl.},
    volume={373},
    date={2011},
    number={2},
    pages={563--579},
}

\bib{BGS1}{article}{
    author={Bongioanni},
    author={Harboure},
    author={Salinas},
    title={Weighted inequalities for commutators of schr\"odinger-Riesz transforms},
    journal={J. Math. Anal. Appl.},
    volume={392},
    date={2012},
    number={1},
    pages={6--22},
}


\bib{Buckley1993}{article}{
   author={Buckley, Stephen M.},
   title={Estimates for operator norms on weighted spaces and reverse Jensen
   inequalities},
   journal={Trans. Amer. Math. Soc.},
   volume={340},
   date={1993},
   number={1},
   pages={253--272}
}


\bib{CD}{article}{
author={Coulhon, Thierry},
author={Duong, Xuan Thinh},
title={Riesz transforms for $1\leq p\leq 2$},
journal={Trans. Amer. Math. Soc.},
volume={351}
date={1999}, 
number={3},
pages={1151--1169},
}

\bib{Cru2015}{article}{
    author={Cruz-Uribe, David},
    title={Two weight norm inequalities for 
    fractional integral operators and commutators},
    date={2015},
    eprint={http://arxiv.org/abs/1412.4157}
}

\bib{CruMarPer2011}{book}{
   author={Cruz-Uribe, David V.},
   author={Martell, Jos{\'e} Maria},
   author={P{\'e}rez, Carlos},
   title={Weights, extrapolation and the theory of Rubio de 
   Francia},
   series={Operator Theory: Advances and Applications},
   volume={215},
   publisher={Birkh\"auser/Springer Basel AG, Basel},
   date={2011},
   pages={xiv+280}
}

\bib{CruMoe2013}{article}{
   author={Cruz-Uribe, David},
   author={Moen, Kabe},
   title={One and two weight norm inequalities for Riesz 
   potentials},
   journal={Illinois J. Math.},
   volume={57},
   date={2013},
   number={1},
   pages={295--323}
}

\bib{DraGraPerPet2005}{article}{
    author={Dragicevi\'c, Oliver},
    author={Grafakos, Loukas},
    author={Pereyra, Mar{\'{\i}}a Cristina},
    author={Petermichl, Stefanie},
    title={Extrapolation and sharp norm estimates for classical operators on
    weighted Lebesgue spaces},
    journal={Publ. Mat.},
    volume={49},
    date={2005},
    number={1}
}




\bib{Duo2001}{book}{
   author={Duoandikoetxea, Javier},
   title={Fourier analysis},
   series={Graduate Studies in Mathematics},
   volume={29},
   note={Translated and revised from the 1995 Spanish original by David
   Cruz-Uribe},
   publisher={American Mathematical Society, Providence, RI},
   date={2001},
   pages={xviii+222},
}


\bib{DM}{article}{
author={Duong, Xuan Thinh}, 
author={MacIntosh, Alan},
title={Singular integral operators with non-smooth kernels on irregular domains},
journal={Rev. Mat. Iberoamericana},
volume={15},
date={1999}, 
number={2}, 
pages={233--265},
}


\bib{DY1}{article}{
author={Duong, Xuan Thinh},
author={Yan, Lixin},
title={New function spaces of BMO type, the John-Nirenberg inequality, interpolation, and applications},
journal={Comm. Pure Appl. Math.},
volume={58},
date={2005},
number={10},
pages={1375--1420},

}


\bib{DY2}{article}{
author={Duong, Xuan Thinh},
author={Yan, Lixin},
title={Duality of Hardy and BMO spaces associated with operators with heat kernel bounds},
journal={J. Amer. Math. Soc.},
volume={18},
date={2005}, 
number={4}, 
pages={943--973},
}

\bib{DZ1}{article}{
 author={Dziuba\'nski},
 author={Zienkiewicz}, 
 title={Hardy space $H^1$ associated to Schr\"odinger operator with potential satisfying reverse H\"older inequality}, 
 journal={Rev. Mat. Iberoamericana},
 volume={15} 
 date={1999}, 
 number={2}, 
 pages={279--296},
}

\bib{DZ2}{article}{
 author={Dziuba\'nski},
 author={Zienkiewicz}, 
 title={$H^p$ spaces for Schr\"odinger operators, Fourier analysis and related topics (Bedlewo, 2000)}, 
 journal={Banach Center Publ., Polish Acad. Sci., Warsaw},
 volume={56} 
 date={2002}, 
 pages={45--53},
}

\bib{DZ3}{article}{
 author={Dziuba\'nski},
 author={Zienkiewicz}, 
 title={$H^p$ spaces associated with Schr\"odinger operators with potentials from reverse H\"older classes}, 
 journal={Colloq. Math.},
 volume={98} 
 date={2003},
 number={1}, 
 pages={5--38},
}


\bib{Gra2004}{book}{
   author={Grafakos, Loukas},
   title={Classical and modern Fourier analysis},
   publisher={Pearson Education, Inc., Upper Saddle River, NJ},
   date={2004},
   pages={xii+931}
}



\bib{HMar}{article}{
author={Hofmann, Steve},
author={Martell, Jos\'e Mar\'ia},
title={$A^\infty$ estimates via extrapolation of Carleson measures and applications to divergence form elliptic operators},
journal={Trans. Amer. Math. Soc.},
volume={364}
date={2012},
number={1},
pages={65--101},


}

\bib{HMay}{article}{
author={Hofmann, Steve},
author={Mayboroda, Svitlana},
title={Hardy and BMO spaces associated to divergence form elliptic operators},
journal={Math. Ann.},
volume={344},
date={2009},
number={1}, 
pages={37--116},
}

\bib{HMM}{article}{
author={Hofmann, Steve},
author={Mayboroda, Svitlana},
author={McIntosh, Alan},
title={Second order elliptic operators with complex bounded measurable coefficients in $L^p$, Sobolev and Hardy spaces},
journal={Ann. Sci. \'Ec. Norm. Sup\'er.},
volume={44}
date={2011}
number={5},
pages={723--800},
} 

\bib{Hyt2012}{article}{
   author={Hyt{\"o}nen, Tuomas P.},
   title={The sharp weighted bound for general Calder\'on-Zygmund operators},
   journal={Ann. of Math. (2)},
   volume={175},
   date={2012},
   number={3},
   pages={1473--1506}
}


\bib{KalVer1999}{article}{
   author={Kalton, N. J.},
   author={Verbitsky, I. E.},
   title={Nonlinear equations and weighted norm inequalities},
   journal={Trans. Amer. Math. Soc.},
   volume={351},
   date={1999},
   number={9},
   pages={3441--3497}
}

\bib{LacMoePerTor2010}{article}{
   author={Lacey, Michael T.},
   author={Moen, Kabe},
   author={P{\'e}rez, Carlos},
   author={Torres, Rodolfo H.},
   title={Sharp weighted bounds for fractional integral operators},
   journal={J. Funct. Anal.},
   volume={259},
   date={2010},
   number={5},
   pages={1073--1097}
 }
 
 \bib{LacPetReg2010}{article}{
   author={Lacey, Michael T.},
   author={Petermichl, Stefanie},
   author={Reguera, Maria Carmen},
   title={Sharp $A_2$ inequality for Haar shift operators},
   journal={Math. Ann.},
   volume={348},
   date={2010},
   number={1},
   pages={127--141}
}
 
\bib{LacSawUri2009}{article}{
  author={Lacey, Michael T.},
  author={Sawyer, Eric T.},
  author={Uriarte-Tuero, Ignacio},
  title={Two-weight Inequalities for Discrete Positive
    Operators},
  eprint={http://arxiv.org/abs/0911.3437},
  date={2009}
}


\bib{LacSpe2015}{article}{
   author={Lacey, Michael T.},
   author={Spencer, Scott},
   title={On entropy bumps for Calder\'on-Zygmund operators},
   journal={Concr. Oper.},
   volume={2},
   date={2015},
   pages={47--52}
}
 
\bib{Moe2012}{article}{
   author={Moen, Kabe},
   title={Sharp weighted bounds without testing or extrapolation},
   journal={Arch. Math. (Basel)},
   volume={99},
   date={2012},
   number={5},
   pages={457--466}
}

\bib{MucWhe1971}{article}{
   author={Muckenhoupt, Benjamin},
   author={Wheeden, Richard L.},
   title={Weighted norm inequalities for singular and fractional integrals},
   journal={Trans. Amer. Math. Soc.},
   volume={161},
   date={1971},
   pages={249--258}
}

\bib{Pet2007}{article}{
   author={Petermichl, S.},
   title={The sharp bound for the Hilbert transform on weighted Lebesgue
   spaces in terms of the classical $A_p$ characteristic},
   journal={Amer. J. Math.},
   volume={129},
   date={2007},
   number={5},
   pages={1355--1375}
}

\bib{Pet2008}{article}{
   author={Petermichl, Stefanie},
   title={The sharp weighted bound for the Riesz transforms},
   journal={Proc. Amer. Math. Soc.},
   volume={136},
   date={2008},
   number={4},
   pages={1237--1249}
}

\bib{PetVol2002}{article}{
   author={Petermichl, Stefanie},
   author={Volberg, Alexander},
   title={Heating of the Ahlfors-Beurling operator: weakly quasiregular maps
   on the plane are quasiregular},
   journal={Duke Math. J.},
   volume={112},
   date={2002},
   number={2},
   pages={281--305}
}

\bib{PotReg2013}{article}{
   author={Pott, Sandra},
   author={Reguera, Maria Carmen},
   title={Sharp B\'ekoll\'e estimates for the Bergman projection},
   journal={J. Funct. Anal.},
   volume={265},
   date={2013},
   number={12},
   pages={3233--3244}
}

\bib{RahSpe2015}{article}{
   author={Rahm, Robert},
   author={Spencer, Scott},
   title={Entropy bump conditions for fractional maximal and integral
   operators},
   journal={Concr. Oper.},
   volume={3},
   date={2016},
   pages={112--121}
}

\bib{Saw1984}{article}{
   author={Sawyer, Eric},
   title={A two weight weak type inequality for fractional integrals},
   journal={Trans. Amer. Math. Soc.},
   volume={281},
   date={1984},
   number={1},
   pages={339--345}
}

\bib{Saw1988}{article}{
   author={Sawyer, Eric T.},
   title={A characterization of two weight norm inequalities for fractional
   and Poisson integrals},
   journal={Trans. Amer. Math. Soc.},
   volume={308},
   date={1988},
   number={2},
   pages={533--545}
}

\bib{Sh} {article}{
author={Shen, Zhongwei},
title={$L^p$ estimates for Schr\"odinger operators with certain potentials},
journal={Ann. Inst. Fourier (Grenoble)},
volume={45}, 
date={1995}, 
pages={513--546}
}

\bib{Tang2011}{article}{
   author={Tang, Lin},
   title={Weighted norm inequalities for Schr\"odinger type operators},
   journal={Forum Math.},
   volume={27},
   date={2015},
   number={4},
   pages={2491--2532}
}

\bib{TreVol2015}{article}{
   author={Treil, Sergei},
   author={Volberg, Alexander},
   title={Entropy conditions in two weight inequalities for singular
   integral operators},
   journal={Adv. Math.},
   volume={301},
   date={2016},
   pages={499--548}
}

\end{biblist}
\end{bibsection}
\end{document}